\documentclass[12pt,a4paper,reqno]{amsart}
\usepackage{amsmath,amsthm,amssymb,mathtools,thmtools,thm-restate}
\usepackage[mathscr]{eucal}
\usepackage{cite}
\usepackage{upgreek}
\usepackage{color}
\usepackage[bookmarks=false]{hyperref}
\usepackage{enumerate,enumitem}
\usepackage[all,cmtip]{xy}

\numberwithin{equation}{section}

\theoremstyle{plain}
\newtheorem{mthm}{Theorem}

\newtheorem{theorem}{Theorem}[section]

\newtheorem{lemma}[theorem]{Lemma}
\newtheorem{proposition}[theorem]{Proposition}
\newtheorem{corollary}[theorem]{Corollary}
\theoremstyle{definition}
\newtheorem{definition}[theorem]{Definition}
\newtheorem*{definition*}{Definition}
\newtheorem{remark}[theorem]{Remark}

\newtheorem{example}[theorem]{Example}

\newcommand{\C}{\mathbb{C}}\newcommand{\cC}{\mathcal{C}}

\newcommand{\bF}{\mathbb{F}}

\newcommand{\cH}{\mathcal{H}}
\newcommand{\cI}{\mathcal{I}}
\newcommand{\cJ}{\mathcal{J}}
\newcommand{\cK}{\mathcal{K}}

\newcommand{\cN}{\mathcal{N}}

\newcommand{\cP}{\mathcal{P}}

\newcommand{\bT}{\mathbb{T}}\newcommand{\cT}{\mathcal{T}}
\newcommand{\cU}{\mathcal{U}}

\newcommand{\Z}{\mathbb{Z}}\newcommand{\cZ}{\mathcal{Z}}
\newcommand{\op}{\mathtt{op}}
\newcommand{\cAP}{\mathcal{AP}}
\newcommand{\cWM}{\mathcal{WM}}
\newcommand{\cQN}{\mathcal{QN}}

\newcommand{\eps}{\varepsilon}

\newcommand{\norm}[1]{\left\|#1\right\|}
\newcommand{\abs}[1]{\left|#1\right|}

\newcommand{\ot}{\otimes}
\newcommand{\otb}{\bar{\otimes}}
\newcommand{\algot}{\otimes_{alg}}

\newcommand{\rarrow}{\rightarrow}

\newcommand{\tnorm}[1]{{\left\vert\kern-0.25ex\left\vert\kern-0.25ex\left\vert #1 
		\right\vert\kern-0.25ex\right\vert\kern-0.25ex\right\vert}}

\DeclareSymbolFont{Symbols}{OMS}{cmsy}{m}{n}
\DeclareMathSymbol{\Emptyset}{\mathord}{Symbols}{"3B}

\newcommand{\act}{\curvearrowright}
\newcommand{\cross}{\rtimes}

\newcommand{\Orb}{\operatorname{Orb}}

\newcommand{\id}{\operatorname{id}}

\renewcommand{\sp}{\operatorname{span}}
\newcommand{\Inv}{\operatorname{Inv}}

\title{On compact extensions of tracial $W^*$-dynamical systems}

\author[A. Jamneshan]{Asgar Jamneshan}
\address{University of Bonn, Mathematical Institute, Endenicher Allee 60, D-53115 Bonn}
\email{ajamnesh@math.uni-bonn.de}

\author[P. Spaas]{Pieter Spaas}
\address{Department of Mathematical Sciences, University of Copenhagen, Universitetsparken 5, DK-2100 Copenhagen \O, Denmark}
\email{pisp@math.ku.dk}

\date{\today}

\allowdisplaybreaks

\begin{document}

\begin{abstract} 
We establish several classification results for compact extensions of tracial $W^*$-dynamical systems and for relatively independent joinings thereof for actions of arbitrary discrete groups. We use these results to answer a question of Austin, Eisner, and Tao and some questions raised by Duvenhage and King. Moreover, combining our results with an earlier classification of weakly mixing extensions by Popa, we can derive  non-commutative Furstenberg--Zimmer type dichotomies on the $L^2$-level. Although in general an adequate generalization of the Furstenberg--Zimmer tower of intermediate compact extensions doesn't seem possible in the von Neumann algebraic framework, we show that there always exists a non-commutative analogue of the finer Host--Kra--Ziegler tower for any ergodic action of a countable abelian group. \end{abstract}

\maketitle

\section{Introduction and statement of the main results}

Throughout this paper, we fix a discrete group $\Gamma$ unless mentioned otherwise. 
A \emph{tracial von Neumann algebra dynamical system} or \emph{tracial $W^*$-dynamical system} is a tuple $(N,\tau,\sigma)$ where $(N,\tau)$ is a tracial von Neumann algebra and $\sigma:\Gamma\act N$ is a trace-preserving action of $\Gamma$ on $N$. In this paper, we do not assume groups to be countable nor von Neumann algebras to be separable unless stated otherwise. If $Q\subset N$ is a $\Gamma$-invariant von Neumann subalgebra, then we say that $N$ is an \emph{extension} of $Q$, and denote this dynamical inclusion by $\Gamma\act (Q\subset N)$.  

For example, consider a commutative tracial von Neumann algebra $(L^\infty(X,\mu),\int)$ where $L^\infty(X,\mu)$ is the algebra of essentially bounded complex functions on a probability space $(X,\mu)$ and the trace is given by integration $\int$. In this context, a trace-preserving action $\sigma:\Gamma\act L^\infty(X,\mu)$ is equivalently described by a measure-preserving action of $\Gamma$ on the probability algebra associated to the probability space $(X,\mu)$, and $\Gamma$-invariant von Neumann subalgebras of $L^\infty(X,\mu)$ uniquely correspond to $\Gamma$-invariant sub-$\sigma$-algebras of the associated probability algebra (e.g., see \cite[\S 7]{jt-foundational}), where a probability algebra is a measure algebra with total mass equal to 1. Structural ergodic theory studies the behavior of measure-preserving group actions on probability algebras relative to $\Gamma$-invariant subalgebras. The aim of this paper is to study the extent to which certain results in structural ergodic theory can be established for general tracial $W^*$-dynamical systems. 
More precisely, we are interested in non-commutative analogues of results in the Furstenberg--Zimmer structure theory and an important refinement of the Furstenberg--Zimmer structure theory in the case that $\Gamma$ is abelian, known as the Host--Kra--Ziegler structure theory\footnote{Furstenberg--Zimmer structure theory was originally developed in the articles \cite{furstenberg-szemeredi,zimmer1,zimmer2}. The interested reader is referred to \cite{einsiedler-ward,furstenberg-book,glasner,kerrli,tao-poincare} for textbook expositions. The Host--Kra--Ziegler structure theory was initiated in \cite{host-kra, ziegler}, see also the textbook exposition \cite{hk-book}.}. 

Roughly speaking, the Furstenberg--Zimmer structure theorem states that any probability algebra dynamical system can be decomposed into a chain of subsystems of a more algebraic nature, called the \emph{distal} part, which are either structured \emph{compact} extensions of earlier subsystems or inverse limits of such extensions, and a final chaotic \emph{weakly mixing} extension of the distal part. In terms of orbital behavior, functions in the distal part of the system have in some precise sense predictable or controlled orbits, while orbits of functions in the weakly mixing part exhibit a strong form of randomness. In the commutative world of probability algebra dynamical systems, the Furstenberg--Zimmer structure theorem is ``universally'' applicable. Indeed, it was originally established for any action of a second countable locally compact group on a countably generated probability algebra by the work of Furstenberg and Zimmer, and has recently been shown to hold for arbitrary actions of a (discrete) group on a probability algebra, independently by Edeko, Haase, and Kreidler in \cite{ehk} and the first author in \cite{jamneshan2019fz}. Furstenberg \cite{furstenberg-szemeredi} originally developed this structure theorem to prove a multiple recurrence theorem for $\mathbb{Z}$-actions, strengthening Poincar\'e's single recurrence theorem, which he famously used to give an ergodic-theoretic proof of Szemer\'edi's theorem \cite{szemeredi}. By analogy with the averages in the classical mean ergodic theorem of von Neumann used for single recurrence, non-conventional ergodic averages correspond to multiple recurrence. In general, it is a difficult problem to prove convergence of non-conventional ergodic averages and to determine the limit. In a breakthrough, Host and Kra \cite{host-kra} and independently Ziegler \cite{ziegler} proved $L^2$-convergence for ergodic $\mathbb{Z}$-actions by classifying the subsystems which govern the limiting behavior. These subsystems are special types of compact extensions, and their deep insight is that these compact extensions can be classified as inverse limits of rotations on nilmanifolds, see \cite{hk-book} for a comprehensive exposition. This refinement of the Furstenberg--Zimmer structure theory is known as the Host--Kra--Ziegler structure theory and has been influential in areas such as combinatorics and number theory, e.g., see \cite{green-tao,gtz}. 

In recent decades, there has been a surge of interest in non-commutative analogues of results in Furstenberg--Zimmer structure theory, character theory, and related topics such as multiple recurrence and convergence of non-conventional ergodic averages. The motivation seems to have been diverse, for example stemming from structure theory of von Neumann algebras (see, e.g., \cite{popa,ioana-arxiv,ioana,chifan-das}), Lie theory and higher rank lattices (see, e.g., \cite{bout1,bout}), non-commutative analogues of results in Furstenberg--Zimmer structure theory and the theory of joinings (see, e.g., \cite{joining-1,joining-2,joining-3,skew-product-1,skew-product-2,duvenhage-joining,duvenhage-joining2,duvenhage-joining3,duvenhage-rel-compact,duvenhage-rel-ergodic,duvenhage-rel-mixing,groh,kerrli,automorphism-quantum,olesen,runde,viselter}), and non-commutative multiple recurrence and convergence of non-conventional ergodic averages (see, e.g., \cite{austin,duvenhage-bergelson,szemeredi-wstar,eisner,fidalea-diagonal,fidaleo-entangled,fidaleo-nonconventional,fidaleo-periodic,niculescu,debeer}).

This paper contributes to this development by providing a comprehensive treatment of compact extensions within the framework of tracial $W^*$-dynamical systems. We now state and discuss our main results and their relation to the existing literature. 

\subsection*{Main results}

Let $\Gamma\act (N,\tau)$ be a tracial $W^*$-dynamical system, and let $Q\subset N$ be a $\Gamma$-invariant von Neumann subalgebra.
We say that $N$ is a \emph{compact extension} of $Q$ if every element of $N$ is \textit{almost periodic relative to $Q$}, that is, its orbit is conditionally precompact, which roughly means that for every $\eps>0$, it is contained within the $\eps$-neighborhood of a finitely generated $Q$-module zonotope of $L^2(N)$ (see Section~\ref{sec-compact} for the precise definition). Our first main result establishes  equivalent characterizations of a compact extension, by analogy with such characterizations in the measure-theoretic framework  (cf. \cite[\S 6.3]{furstenberg-book}, \cite[Theorem 4.1]{jamneshan2019fz}). It connects the above definition using relatively almost periodic elements with conditional Hilbert--Schmidt operators (item (2) below, see Subsection~\ref{ssec:HS} for the definitions), and finitely generated modules (item (3) below, see also Subsection~\ref{ssec:modules}). 

\begin{restatable}{mthm}{compchar}
\label{thm:compchar}
Let $\Gamma\act (N,\tau)$ be a tracial $W^*$-dynamical system, and let $Q\subset N$ be a $\Gamma$-invariant von Neumann subalgebra. 
Then the following are equivalent. 
	\begin{enumerate}
		\item $N$ is a compact extension of $Q$.
		\item $\mathrm{span}\{K\ast_Q\xi\mid K\in HS_Q\; \Gamma\text{-equivariant}, \xi\in L^2(N)\}$ is dense in $L^2(N)$.
		\item $L^2(N)$ is the closure of the union of the $Q$-finitely generated $\Gamma$-invariant right (or left) $Q$-submodules.
	\end{enumerate}
\end{restatable}

Different notions of relatively almost periodic elements and of compact extensions for $W^*$-dynamical systems have appeared in the literature. For example, our definition of relatively almost periodic elements (Definition \ref{def:rap}) is used in this form in \cite{chifan-das}, where it is used to establish classification results for intermediate von Neumann subalgebras, and it was already indicated in \cite[Remark 3.16]{kerrli} (which in turn was inspired by the operator-theoretic conceptualization in the commutative case in \cite[\S 2.13]{tao-poincare}). The relation with conditional Hilbert--Schmidt operators is written out for commutative von Neumann algebras in \cite[Appendix C]{kerrli} (see also Section~\ref{ssec:HS} below), and it is remarked there that a similar relation should hold for general von Neumann algebras. A definition of compact extensions in terms of finitely generated submodules (see item (iii) above) appears in \cite[\S 4]{austin} and a similar definition is given in \cite[\S 3]{duvenhage-rel-compact}. Nevertheless, even though parts of Theorem \ref{thm:compchar} were probably known to experts, a more comprehensive treatment of relative compactness similar to the commutative case, including establishing the equivalence of these definitions, seems to be missing in the literature to the best of our knowledge. 

In previous functional-analytic approaches to Furstenberg--Zimmer structure theory, in particular the part of the theory dealing with compact extensions, an important role is usually played by the so-called ``conditional $L^2$-spaces'', see \cite{tao-poincare,kerrli,jamneshan2019fz,ehk} where these spaces also go by the names of ``conditional Hilbert modules'' or ``Hilbert--Kaplansky modules''. The proof of Theorem \ref{thm:compchar} in Section \ref{sec-compact}, shows that one can avoid introducing such conditional spaces, and we work directly on the usual $L^2$-spaces. 
Finally, we also note that Theorem \ref{thm:compchar} is established for arbitrary discrete groups acting on general von Neumann algebras, and hence our results add to most known partial results by removing commutativity and countability/separability assumptions.

We use our classification of compact extensions from Theorem \ref{thm:compchar} to answer \cite[Question~4.3]{austin} by Austin, Eisner, and Tao about finite dimensional approximations of $Q$-submodules of finite lifted trace, see  Lemma \ref{lem-aet}.  We also answer two questions raised by Duvenhage and King in \cite[\S 6]{duvenhage-rel-compact} about the classification of intermediate subsystems of compact extensions and characterizations of ``generalized eigenfunctions'', see Remark \ref{rem-duvenhage}.   

\subsection*{Addendum} After the first version of our paper appeared on the arXiv, the authors learned\footnote{We thank Terry Tao for bringing this preprint to our attention.} of the unpublished preprint \cite{Jolissaint}, in which Question~4.3 by Austin, Eisner, and Tao from \cite{austin} is also answered with a direct hands-on proof. Our proof is different, and our answer to the question follows immediately from a more general classification result, which we believe puts it in an appropriate framework.

Having characterized compact extensions, a next natural step, as in classical Furstenberg--Zimmer theory, is to try to study the maximal compact extension inside an arbitrary extension. It is easy to verify that the space of conditionally almost periodic elements in $L^2(N)$ is always a closed Hilbert subspace. Thus a natural candidate for the maximal compact extension of $Q$ inside $N$ is the intersection of this closed Hilbert subspace with $N$. In commutative ergodic theory, this intersection corresponds to a $\Gamma$-invariant von Neumann subalgebra of $L^\infty(X,\mu)$. However, for general $W^*$-dynamical systems this intersection is only a strong operator topology closed linear subspace of $N$ that may not be closed under multiplication and/or involution, that is, it may not be a von Neumann subalgebra. This was observed by Austin, Eisner and Tao in \cite[Example 4.2]{austin}, and we recall their example in Section \ref{sec-compact}. Given this obstruction, one could instead try to consider the von Neumann algebra generated by the relatively almost periodic elements as suggested in \cite[Remark 3.16]{kerrli}, see also \cite{chifan-peterson}. However, the example of Austin, Eisner, and Tao shows that this von Neumann subalgebra can contain all relatively weakly mixing elements. It thus seems that this is not a satisfactory solution in general. 

Nevertheless, under additional hypotheses on the inclusion $Q\subset N$, it is possible to identify the space of relatively almost periodic elements with a von Neumann subalgebra. A natural such hypothesis is quasi-regularity of the inclusion which has been put forward by Chifan and Peterson in \cite{chifan-peterson}. We recall its definition and provide a proof as to its effect in Section \ref{sec-compact} for the convenience of the reader. 

Quasi-normalizers had been connected to Furstenberg--Zimmer structure theory before in \cite[\S 6]{ioana-arxiv}, where Ioana shows that for an ergodic action $\Gamma\act (X,\mu)$ of a countable discrete group on a standard probability space $(X,\mu)$, the quasi-normalizer of $L(\Gamma)$ inside $L^\infty(X)\cross \Gamma$ is exactly $L^\infty(Y)\cross\Gamma$, where $Y$ is the maximal compact subsystem of $X$. This can be used to give a purely von Neumann algebraic description of the Furstenberg--Zimmer tower in the commutative setting (see \cite{chifan-peterson}, see also \cite[Theorem 2.3.]{chifan-das}). Further connections with the quasi-normalizers of the involved algebras in the $W^*$-dynamical setting are also explored in \cite{chifan-peterson}. 

Alternatively, if one does not insist on the maximality of the compact extension, one could ask if there are other natural intermediate compact extensions which correspond to von Neumann subalgebras. In the commutative context, a prominent and useful example of a tower of intermediate compact extensions is formed by the Host--Kra--Ziegler factors \cite{host-kra,ziegler}, when one considers actions of abelian groups. We will show in Section \ref{sec-hk} that there is a natural non-commutative analogue of these factors for abelian group actions, leading to a tower of von Neumann subalgebras corresponding to intermediate compact extensions (without additional hypotheses on the subalgebras).

Before that, we study relatively independent joinings of tracial $W^*$-dynamical systems in Section \ref{sec:joingings}. In particular, we classify the invariant elements of relatively independent joinings  (see Proposition \ref{prop:invJoin}). Moreover, we extend an important theorem of Furstenberg \cite[Theorem~7.1]{furstenberg-szemeredi} on relatively independent joinings of compact extensions over a common factor (see also \cite[Theorem~9.21]{glasner}) to $W^*$-dynamical central extensions as follows. 

\begin{mthm}
\label{thm:RelProd}
Let $(N_1,\tau_1,\sigma_1)$ and $(N_2,\tau_2,\sigma_2)$ be $W^*$-dynamical $\Gamma$-systems and let $Q$ be a common $\Gamma$-invariant von Neumann subalgebra contained in their respective centers where the actions on $Q$ agree. For $i=1,2$, denote by $P_i\subset N_i$ the maximal compact extension of $Q$ inside $N_i$. 
Consider the extension $\Gamma\act (Q\subset N_1\otb_Q N_2^\op)$, where $N_1\otb_Q N_2^\op$ denotes the relatively independent product of $N_1,N_2$ over $Q$, see Definition \ref{def-relprod}. Then the maximal compact extension $P$ of $Q$ inside $N_1\otb_Q N_2^\op$ is given by $P=P_1\otb_Q P_2^\op$. 
\end{mthm}

Several of the results developed so far are used in the proof of Theorem \ref{thm:RelProd}, in particular the characterization using conditional Hilbert--Schmidt operators from Theorem~\ref{thm:compchar}(2). Although our proof is inspired by Furstenberg's original proof, we provide a purely algebraic and streamlined version of the proof. In particular, our proof also works without any additional countability/separability assumptions.

Combining our results on relatively independent joinings and compact extensions with Popa's classification of weakly mixing extensions \cite[\S 2]{popa}, we analyse the dichotomy between relatively weakly mixing and relatively compact $W^*$-dynamical systems in Section~\ref{sec-dichotomy}. We prove a relative dichotomy on the $L^2$-level for arbitrary extensions of $W^*$-dynamical systems which can be reduced to the von Neumann algebra level if the inclusion is quasi-regular. 
By the previously mentioned example of Austin, Eisner, and Tao we cannot expect a relative dichotomy to hold on the von Neumann algebra level without further assumptions on the inclusion. Moreover, quasi-regularity is not necessarily preserved under taking the maximal compact extension (see Remark \ref{ex:qreg}). 
These are therefore serious obstructions to establishing a complete analogue of the Furstenberg--Zimmer structure theorem for general $W^*$-dynamical systems, beyond the aforementioned dichotomies. Nevertheless, as mentioned earlier, some alternative results in this direction, exploiting the quasi-normalizers of the involved algebras, are established in \cite{chifan-peterson}.

We continue Section~\ref{sec-dichotomy} by extending a classical characterization of weakly mixing extensions in terms of joinings to the von Neumann algebraic framework. Namely, under the assumption of quasi-regularity, we characterize weakly mixing extensions of $W^*$-dynamical systems as being those that are disjoint from compact extensions, see Proposition~\ref{thm-disjoint}. We point out that our positive results in this Section generalize, refine, and provide new proofs of related results in \cite{joining-1,duvenhage-joining2,groh,olesen}.  

Finally, assuming now that $\Gamma$ is countable and abelian, we develop a non-commutative version of the Host--Kra cubic systems and the Gowers--Host--Kra seminorms for ergodic $W^*$-dynamical systems in Section \ref{sec-hk}. Using our previous results on joinings, we establish the following non-commutative Host--Kra--Ziegler tower of compact extensions:  

\begin{mthm}\label{thm-hk}
Let $\Gamma$ be a countable discrete abelian group, $(N,\tau)$ a tracial von Neumann algebra with separable predual, and let $(N,\tau,\sigma)$ be an ergodic $W^*$-dynamical $\Gamma$-system.  
Then there exists an increasing sequence $Z_1\subset Z_2\subset \ldots \subset N$ of inclusions of $W^*$-systems with the following properties. 
\begin{itemize}
    \item[(i)]  $Z_1$ is the maximal compact $W^*$-subsystem of $N$. 
    \item[(ii)] For each $n\geq 1$, $Z_{n+1}$ is a compact extension of $Z_n$. 
    \item[(iii)] For each $n\geq 1$, $x\in Z_n$ if and only if $\tnorm{x}_{n+1}=0$, where $\tnorm{\cdot}_{n+1}$ denotes the non-commutative version of the Gowers--Host--Kra uniformity seminorms, see Definition \ref{def-seminorm}. 
\end{itemize}
\end{mthm}

We believe this non-commutative generalization of the Host--Kra--Ziegler hierarchy to be a promising new direction in the study of $W^*$-dynamical systems, and we plan to further investigate the compact extensions arising in Theorem~\ref{thm-hk} in future work. In particular, we plan to establish finer algebraic and geometric descriptions of them by analogy with the commutative theory for $\mathbb{Z}$-actions and relate them with non-commutative multiple recurrence and non-commutative non-conventional ergodic averages. 

\subsection*{Organization of the paper} Besides the introduction, this paper has five other sections. In the Preliminaries, we recall some definitions, constructions, and tools needed in the remainder of the paper, and we prove some useful lemmas. In Section~\ref{sec-compact}, we give a precise definition of a compact extension of $W^*$-dynamical systems. We then prove our first main classification Theorem~\ref{thm:compchar}, derive some corollaries, answer the previously mentioned questions from the literature, and discuss a few related results.
We continue in Section~\ref{sec:joingings} with the study of joinings of $W^*$-dynamical systems, including a proof of Theorem~\ref{thm:RelProd}. In Section~\ref{sec-dichotomy}, we investigate the dichotomy between weakly mixing and compact extensions. Finally, in Section~\ref{sec-hk}, we extend the notions of Host--Kra--Ziegler factors and Gowers--Host--Kra uniformity seminorms to the von Neumann algebraic framework and prove Theorem~\ref{thm-hk}.

\subsection*{Acknowledgements} The first author thanks Terence Tao for encouragement. The second author thanks Adrian Ioana for some stimulating discussions, and Ionu\c{t} Chifan for a useful discussion on results related to \cite{chifan-peterson}. The authors would also like to thank Cyril Houdayer for kindly pointing out an issue with Definition~\ref{def:joining} in an earlier version of the paper. The authors are grateful to the anonymous referee for helpful corrections and suggestions.

\section{Preliminaries}\label{sec-prelim}

\subsection{von Neumann algebras}\label{sec-prelim-algebras}

In this section, we fix some general notations for tracial von Neumann algebras. 
For basic results and any undefined concepts in the theory of tracial von Neumann algebras, we refer the interested reader to \cite{AP}. 

Let $\cH$ be a Hilbert space. Von Neumann algebras are unital $^*$-subalgebras of the $^*$-algebra $B(\cH)$ of bounded operators on $\cH$ that are closed in the weak (or equivalently strong) operator topology. Von Neumann's bi-commutant theorem tells us that von Neumann algebras are exactly the $^*$-subalgebras $M$ that are equal to their double commutant $M''$. This equivalence gives rise to a useful interplay between the algebraic and analytic aspects of these objects. 

Unless explicitly stated otherwise, all von Neumann algebras we will consider are \textit{tracial von Neumann algebras}: von Neumann algebras $M$ equipped with a faithful normal tracial state $\tau$, that is, a positive linear functional $\tau:M\rarrow\C$ satisfying $\tau(1)=1$, $\tau(xy)=\tau(yx)$ for all $x,y\in M$, $\tau(x^*x)=0$ iff $x=0$ for all $x\in M$, and $\tau$ is continuous with respect to the weak operator topology when restricted to the unit ball of $M$. 
When necessary, we also write $\tau_M$ to emphasize the von Neumann algebra $M$ on which the trace is defined.  
We denote by $\norm{x}_2 \coloneqq \sqrt{\tau(x^*x)}$ the \textit{2-norm} of $x\in M$ and denote by $L^2(M)$ the completion of $M$ with respect to this norm. When identifying $M$ with a dense subspace of $L^2(M)$, we will also denote it by $\hat M$, and we will write $\hat x\in L^2(M)$ when $x\in M$. 

Let $P\subset M$ be an inclusion of von Neumann algebras, which unless stated otherwise is assumed to be unital (that is, $1_P = 1_M$). We denote by $E_P:M\rarrow P$ the unique $\tau$-preserving conditional expectation from $M$ onto $P$, which arises as the restriction to $M$ of the orthogonal projection from $L^2(M)$ onto $L^2(P)$. We denote by $P'\cap M \coloneqq \{x\in M\mid \forall y\in P: xy=yx\}$ the \emph{relative commutant} of $P$ in $M$, and by $\cN_M(P)\coloneqq \{u\in\cU(M)\mid uPu^*=P\}$ the (set-wise) \emph{normalizer} of $P$ in $M$. We say that $P$ is \emph{regular} in $M$ if $\cN_M(P)'' = M$, that is, $\cN_M(P)$ generates $M$ as a von Neumann algebra.
We will further denote by $\cU(M)$ the unitary group of $M$ and by $\cP(M)$ the set of projections in $M$.

\subsection{Hilbert modules and conditionally orthonormal bases}\label{ssec:modules}

For the reader's convenience, we will recount here some of the basic properties and constructions which we will need later on. A von Neumann algebraist may want to skip this section and refer back to the notation if necessary.

\begin{definition}
Let $M$ be a von Neumann algebra. A \textit{left $M$-module} is a Hilbert space $\cH$ together with a normal $^*$-homomorphism $\pi: M\rarrow B(\cH)$. Similarly, a \textit{right $M$-module} comes with a normal $^*$-homomorphism $\pi: M^{\text{op}}\rarrow B(\cH)$, where $M^{\text{op}}$ denotes the opposite von Neumann algebra, that is, $M$ equipped with reversed multiplication operation  $x\cdot y \coloneqq yx$.
\end{definition}

We will mostly work with right modules, though we note that this is just a choice, and we could as well interchange the roles of left and right throughout the paper.

\begin{example}
The most basic example of a right $M$-module is $L^2(M)$, where $M$ acts by right multiplication. Similarly, given any Hilbert space $\cH$, we can let $M$ act by right multiplication on $L^2(M)\ot \cH$ (i.e., right multiplication on $L^2(M)$, and trivially on $\cH$). Given a projection $p\in B(L^2(M)\ot\cH)$ which is invariant under right multiplication by $M$, we also obtain in this way a right $M$-module of the form $p(L^2(M)\ot\cH)$. 
In fact, every right $M$-module is isomorphic to a module of this form, see for instance \cite[Proposition 8.2.3]{AP}. 
\end{example}

\begin{definition}\label{def-leftbdd}
Let $(M,\tau)$ be a tracial von Neumann algebra and $\cH$ a right Hilbert $M$-module. 
We say that a vector $\xi\in\cH$ is \textit{left ($M$-)bounded} if there exists $c>0$ such that $\norm{\xi x} \leq c\norm{x}_2$ for every $x\in M$, that is, the map $\hat x\mapsto \xi x$ extends to a bounded operator $L_\xi$ from $L^2(M)$ to $\cH$. We denote by $\cH^0$ the subspace of left bounded vectors in $\cH$.  
\end{definition}

It is known (see for instance \cite[Proposition 8.4.4]{AP}) that $\cH^0$ is always dense in $\cH$. 
Observe that given $\xi,\eta\in \cH^0$, the operator $L_\xi^*L_\eta$ commutes with the right $M$-action on $L^2(M)$, and thus equals left multiplication by an element from $M$ (see for instance \cite[Theorem 7.1.1]{AP}). Upon identifying $L_\xi^*L_\eta$ with this element, we thus get $L_\xi^*L_\eta\in M$. We can thus define an \textit{$M$-valued inner product} on $\cH^0$ by $$\langle \xi,\eta\rangle_M\coloneqq L_\xi^*L_\eta\in M$$  
The following straightforward lemma shows that this $M$-valued inner product indeed behaves like an inner product with values in $M$ (see also \cite[Lemma 8.4.5]{AP}).

\begin{lemma}
   Let $(M,\tau)$ be a tracial von Neumann algebra and $\cH$ a right Hilbert $M$-module. For all $x\in M$ and $\xi,\eta\in \cH^0$, the following properties hold. 
    \begin{enumerate}
        \item $\langle \xi,\xi\rangle_M\geq 0$ and $\langle \xi,\xi\rangle_M=0$ if and only if $\xi=0$.
        \item $(\langle \xi,\eta\rangle_M)^* = \langle \eta,\xi\rangle_M$.
        \item $\langle \xi,\eta x\rangle_M = \langle \xi,\eta\rangle_M x$ and $\langle \xi x,\eta\rangle_M = x^* \langle \xi,\eta\rangle_M$.
        \item $\tau(\langle \xi,\eta\rangle_M) = \langle \xi,\eta\rangle_\cH$.
    \end{enumerate}
\end{lemma}

\begin{example}\label{ex:PsubM}
	When given an inclusion of tracial von Neumann algebras $P\subset M$, we can naturally view $L^2(M)$ as a right (or left) $P$-module. In this case, the $P$-valued inner product is determined by
	\[
	\langle x,y \rangle_P = E_P(x^*y)
	\]
	for $x,y\in M$.
\end{example}

One of the features of this $M$-valued inner product is that it gives us a notion of orthogonality (relative to $M$), and therefore a notion of conditionally orthonormal basis: 

\begin{definition}
 Let $(M,\tau)$ be a tracial von Neumann algebra and $\cH$ a right Hilbert $M$-module. 
  A \textit{conditionally orthonormal basis}, or a \textit{Pimsner--Popa basis}, for $\cH$ is a family $(\xi_i)_{i\in I}$ of non-zero left bounded vectors such that   
  \begin{itemize}
    \item[(i)] $\overline{\sum_i \xi_i M} = \cH$, and
    \item[(ii)] for all $i,j\in I$: $\langle \xi_i,\xi_j\rangle_M = \delta_{i,j} p_i$ for some projections $p_i\in\cP(M)$.
\end{itemize}
If $(\xi_i)_{i\in I}$ satisfies (ii), but not necessarily (i), we call the family $(\xi_i)_{i\in I}$ \emph{conditionally orthonormal}.  
If $\cH$ has a finite orthonormal basis as a module over $M$, then we also say that $\cH$ is \emph{finite rank}. 
\end{definition}

\begin{remark}
    A standard Gram-Schmidt orthonormalization shows that any finitely generated module has a finite orthonormal basis, and is thus a finite rank module, see for instance \cite[Lemma~8.5.2]{AP}.
\end{remark}

We also record here the following useful proposition, characterizing finite rank modules:

\begin{proposition}\label{prop:finiterank}
    Let $\cH$ be a right $M$-module and $(\xi_i)_{i\in I}$ be a conditionally orthonormal basis. Then $\cH$ is finite rank if and only if there exists $C>0$ such that $\sum_{i\in I} E_\cZ(\langle \xi_i,\xi_i\rangle_M) \leq C 1_M$, where $\cZ$ denotes the center of $M$ and $E_\cZ$ denotes its center-valued trace. 
\end{proposition}
\begin{proof}
This is exactly \cite[Proposition~9.3.2(i)]{AP} once one observes that $\hat{E}_Z(1)$ as defined in that proposition equals $\sum_{i\in I} E_\cZ(\langle \xi_i,\xi_i\rangle_M)$. Indeed, by \cite[Lemma~8.4.8]{AP}, $\sum_{i\in I} L_{\xi_i} L_{\xi_i}^* = 1$, and as in the beginning of \cite[\S 9.3]{AP}, we have $\hat E_\cZ(TT^*) = E_\cZ(T^*T)$ for any bounded, right $M$-linear operator $T: L^2(M) \to \cH$. Since $\langle \xi_i,\xi_i\rangle_M = L_{\xi_i}^*L_{\xi_i}$, the claim thus follows immediately.
\end{proof}

We refer the interested reader to \cite[Sections~8.4 and 9.3]{AP} for more details and properties of orthonormal bases and $M$-modules.

\subsection{Bimodules}\label{sec:bimod}

\begin{definition}
Let $(M_1,\tau_1)$ and $(M_2,\tau_2)$ be two tracial von Neumann algebras. An \textit{$M_1$-$M_2$-bimodule} is a Hilbert space $\cH$ endowed with two normal, commuting $^*$-homomorphisms $\pi_1:M_1\rightarrow B(\cH)$ and $\pi_2:M_2^{\text{op}}\rightarrow B(\cH)$. 
In this case, we define a $^*$-homomorphism $\pi_{\cH}:M_1\otimes M_2^{\text{op}}\rightarrow B(\cH)$ by $\pi_{\cH}(x\otimes y^{\text{op}}) = \pi_1(x)\pi_2(y^{\text{op}})$ and write $x\xi y=\pi_1(x)\pi_2(y^{\text{op}})\xi$, for all $x\in M_1$, $y\in M_2$ and $\xi\in\cH$. 
We also write $_{M_1}{\cH}_{M_2}$ to indicate that $\cH$ is an $M_1$-$M_2$-bimodule. 
\end{definition}

Examples of bimodules include the {\it trivial} $M_1$-bimodule $_{M_1}L^2(M_1)_{M_1}$ and the {\it coarse} $M_1$-$M_2$-bimodule $_{M_1}L^2(M_1)\otimes L^2(M_2)_{M_2}$.

Next, we recall an important construction for bimodules. If $\cH$ is an $M_1$-$M_2$-bimodule and $\cK$ is an $M_2$-$M_3$-bimodule, then the \emph{Connes fusion tensor product}  of $\cH$ and $\cK$ is an $M_1$-$M_3$-bimodule denoted by $\cH\otimes_{M_2}\cK$. 
In order to define this tensor product, one checks (see for instance \cite[Proposition 13.2.1]{AP}) that the formula
\[
\langle \xi_1\ot\eta_1,\xi_2\ot\eta_2\rangle \coloneqq \langle \eta_1,\langle \xi_1,\xi_2\rangle_{M_2} \eta_2\rangle_\cK
\]
yields a positive sesquilinear form on the algebraic tensor product $\cH^0\odot \cK$. We then define $\cH\ot_{M_2}\cK$ as the separation and completion of $\cH^0\odot\cK$ with respect to this sesquilinear form. It is straightforward to check that this Hilbert space becomes an $M_1$-$M_3$-bimodule for the canonical actions given by
\[
x(\xi\ot_{M_2}\eta) = (x\eta)\ot_{M_2}\eta, \quad \text{and} \quad (\xi\ot_{M_2}\eta)y = \xi\ot_{M_2}(\eta y),
\]
for $x\in M_1$, $y\in M_3$, $\xi\in \cH^0$, $\eta\in K$. Moreover, the resulting tensor product operation on bimodules is easily seen to be associative. For further details, we refer the interested reader to e.g. \cite[Appendix B.$\delta$ of Chapter 5]{co94} or \cite[Chapter 13]{AP}.

Next, fix two von Neumann algebras $M$ and $N$. An $M$-$N$-bimodule $\cH$ together with a fixed vector $\xi\in\cH$ is also called a \emph{pointed $M$-$N$-bimodule}.  It is well-known (see, e.g., \cite[Section 13.1.2]{AP}) that there is a one-to-one correspondence between normal completely positive maps $\Phi:M\rarrow N$ and pointed $M$-$N$-bimodules $(\cH,\xi)$ with $\xi$ a left $N$-bounded and cyclic vector in $\cH$, that is, $\overline{\sp\{M\xi N\}}=\cH$. We briefly recall the construction here for later reference.

First, let $\Phi:M\rarrow N$ be a normal completely positive map. On the algebraic tensor product $M\odot N^{\op}$, we consider the positive sesquilinear functional 
\[
\langle x_1\ot y_1,x_2\ot y_2\rangle_\Phi = \tau_N(y_1^*\Phi(x_1^*x_2)y_2)
\]
where $x_1,x_2\in M$, $y_1,y_2\in N$. Let $\cH_\Phi$ denote the completion of the quotient of $M\odot N^{\op}$ by the kernel of this functional, which becomes an $M$-$N$-bimodule for the canonical actions by left, resp. right, multiplication. Note that $\cH$ comes equipped with a special vector, namely (the class of) $1_M\ot 1_N\eqqcolon \xi_\Phi$, which is easily seen to be cyclic and left $N$-bounded.

Conversely, given a pointed $M$-$N$-bimodule $(\cH,\xi)$, where $\xi$ is cyclic and left $N$-bounded, we can build a completely positive map as follows. Consider the bounded operator $L_\xi: L^2(N)\rarrow \cH$ from Definition \ref{def-leftbdd}, and for $x\in M$ define $\Phi_\cH(x)\coloneqq L_\xi^* \pi_M(x) L_\xi$. It is then easily checked that $\Phi_\cH$ takes values in $N$, is completely positive, and that $\Phi\mapsto (\cH_\Phi,\xi_\Phi)$ and $(\cH,\xi)\mapsto \Phi_\cH$ are inverse to each other.

\subsection{Relatively independent products}\label{ssec:relprod}

In this section, we briefly recall the classical construction of relatively independent products for later reference, and rewrite their $L^2$-space in purely algebraic terms.  
We will recall the constructions in the case of standard probability spaces where a classical disintegration is available. 
Similar constructions are available for general probability algebras by working with canonical disintegrations (e.g., see  \cite[\S 2,3]{jamneshan2019fz} and the references therein). 
In the next section, we will discuss the analogous $L^2$-spaces in the non-commutative setting, discuss conditional Hilbert--Schmidt operators, and relate them to these $L^2$-spaces.

Suppose $(X,\mu)$ and $(Y,\nu)$ are standard probability spaces, and let $\pi:X\rarrow Y$ be a factor map. We then have a classical disintegration 
\[
\mu = \int_Y d\mu_y\,d\nu(y)
\]
where $\mu_y$ is a probability measure supported on $\pi^{-1}(\{y\})$. We can now construct a relatively independent product $X\times_Y X$ as follows. As a set, take 
\[
X\times_Y X = \{(x_1,x_2)\in X\times X\mid \pi(x_1) = \pi(x_2)\},
\]
and equip it with the measure
\[
\mu\times_Y\mu = \int_Y (\mu_y\times \mu_y)\,d\nu(y).
\]
It is then easy to verify that this relatively independent product fits in the commutative diagram
\[
\xymatrix{ & X\times_Y X \ar[ld]_{\Pi_1} \ar[rd]^{\Pi_2} & \\
X \ar[rd]_{\pi} & & X \ar[ld]^{\pi}\\
& Y &}
\]
for the coordinate projection maps $\Pi_1:X\times_Y X\rarrow X$ and $\Pi_2:X\times_Y X\rarrow X$ onto the first and second copy of $X$ respectively. Note that we have
\[
\int_{X\times_Y X} f_1(x)f_2(x') \,d(\mu\times_Y\mu)(x,x') = \int_Y \left(\int_{X\times X} f_1(x)f_2(x') \,d(\mu_y\times \mu_y(x,x'))\right)d\nu(y).
\]
From the pointless view on the level of the $L^2$-spaces, the above gives the commutative diagram of inclusions
\[
	\xymatrix{ & L^2(X\times_Y X) & \\
		L^2(X) \ar[ru]^{\iota_1} & & L^2(X) \ar[lu]_{\iota_2}\\
		& L^2(Y) \ar[lu]^{\pi_*} \ar[ru]_{\pi_*} &}
\]
Note that all the above $L^2$-spaces are canonically (left and right) $L^\infty(Y)$-modules. Moreover, it is not hard to see that
\begin{equation}\label{eq:RelProdL2}
L^2(X\times_Y X)\cong L^2(X)\ot_{L^\infty(Y)} L^2(X)
\end{equation}
where the right hand side denotes the Connes fusion tensor product of $L^2(X)$ as a right $L^\infty(Y)$-module with $L^2(X)$ as a left $L^\infty(Y)$-module. Note that in this commutative situation, $L^2(X)\ot_{L^\infty(Y)} L^2(X)$ is still canonically an $L^\infty(Y)$-module. 

To check \eqref{eq:RelProdL2}, we observe that the $L^\infty(Y)$-valued inner product on $L^\infty(X)$ is given by
\[
\langle \xi,\eta\rangle_{L^\infty(Y)}= E_{L^\infty(Y)}(\overline{\xi}\eta)\in L^\infty(Y),
\]
for $\xi,\eta\in L^\infty(X)$ (cf. Example~\ref{ex:PsubM}). 
One can then establish an isomorphism
\begin{equation}\label{eq:HSiso}
L^2(X)\ot_{L^\infty(Y)} L^2(X) \xrightarrow{\cong} L^2(X\times_Y X)
\end{equation}
by
\[
\xi\ot_{L^\infty(Y)} \eta \mapsto \zeta(x_1,x_2) = \xi(x_1)\eta(x_2).
\]
This map is obviously surjective, and the following calculation implies that it is well-defined and injective.
\begin{equation*}
\begin{split}
\langle \xi\ot_{L^\infty(Y)}\eta, &\xi\ot_{L^\infty(Y)}\eta\rangle_{L^2(X)\ot_{L^\infty(Y)} L^2(X)}\\ &= \langle \eta,\langle \xi,\xi\rangle_{L^\infty(Y)} \eta\rangle_{L^2(X)} \\
&= \langle \eta, E_{L^\infty(Y)}(\abs{\xi}^2)\eta\rangle_{L^2(X)}\\
&= \int_X E_{L^\infty(Y)}(\abs{\xi}^2)(x) \abs{\eta}^2(x)\,d\mu(x)\\
&= \int_Y \left(\int_X E_{L^\infty(Y)}(\abs{\xi}^2)(x) \abs{\eta}^2(x)\,d\mu_y(x)\right) d\nu(y)\\
&= \int_Y \left(\int_X  E_{L^\infty(Y)}(\abs{\xi}^2)(y) \abs{\eta}^2(x)\,d\mu_y(x)\right) d\nu(y)\\
&= \int_Y \left(\int_X \left[\int_X \abs{\xi}^2(x')\,d\mu_{y}(x')\right] \abs{\eta}^2(x)\,d\mu_y(x)\right) d\nu(y)\\
%&= \int_Y \left(\int_X \left[\int_X \abs{\xi}^2(x')\abs{\eta}^2(x)\,d\mu_{y}(x')\right] \,d\mu_y(x)\right) d\nu(y)\\
&= \int_Y \left(\int_{X\times X} \abs{\xi}^2(x')\abs{\eta}^2(x)\,d(\mu_{y}\times \mu_y(x,x'))\right) d\nu(y)\\
&= \int_{X\times_Y X} \abs{\xi}^2(x')\abs{\eta}^2(x)\,d(\mu\times_Y \mu)(x',x)\\
&= \langle \zeta, \zeta\rangle_{L^2(X\times_Y X)}.
\end{split}
\end{equation*}

\subsection{Conditional Hilbert--Schmidt operators}\label{ssec:HS} Our treatment of conditional Hilbert--Schmidt operators is motivated by the ones given in \cite[\S 3.3.2]{petersonnotes} and \cite[Appendix~C]{kerrli}, though some care has to be taken since the $Q$-valued inner product is only defined on the dense subspace of left bounded vectors in general, and we can only make sense of conditional orthonormality on this dense subspace. 

Before continuing, we recall the basic construction for inclusions of tracial von Neumann algebras.

\begin{definition}
    Let $(N,\tau)$ be a tracial von Neumann algebra and $Q\subset N$ be a von Neumann subalgebra. Denote by $e_Q\in B(L^2(N))$ the orthogonal projection onto $L^2(Q)$. The von Neumann algebra $\langle N,e_Q\rangle$ generated by $N$ and $e_Q$ inside $B(L^2(N))$ is called the \textit{basic construction} for the inclusion $Q\subset N$.
\end{definition}

It is well-known that $\langle N,e_Q\rangle$ consists exactly of the operators in $B(L^2(N))$ which commute with the right $Q$-action, see for instance \cite[Proposition~9.4.2]{AP}.

We can now define conditional Hilbert--Schmidt operators in the above framework. This definition appears for instance in \cite[Definition~3.6]{petersonnotes} in the commutative framework.

\begin{definition}\label{def:HS}
Let $(N,\tau)$ be a tracial von Neumann algebra and $Q\subset N$ a von Neumann subalgebra. An operator $T\in B(L^2(N))$ is \textit{Hilbert--Schmidt relative to $Q$} (or, \textit{conditionally Hilbert--Schmidt} if $Q$ is clear from context) if $T\in \langle N,e_Q\rangle$, i.e., $T$ is in the basic construction of $Q\subset N$, and $\sum_{\xi\in\Omega} \norm{T\xi}_2^2 < \infty$ for every conditionally orthonormal set $\Omega\subset L^2(N)^0$ where we view $L^2(N)$ as a right $Q$-module. 

We will denote by $HS_{N,Q}$, or $HS_Q$, the set of conditional Hilbert--Schmidt operators, and we equip it with the Hilbert--Schmidt norm given by $\norm{T}_{HS}^2\coloneqq \sum_i \norm{Te_i}_2^2$, where $\{e_i\}_i$ is any conditionally orthonormal basis of $L^2(N)$ as a right $Q$-module.
\end{definition}

We note that it follows from Lemma~\ref{lem:HS} below that the Hilbert--Schmidt norm defined above is independent of the chosen basis.

\begin{remark}
One can define conditional Hilbert--Schmidt operators more generally between any two right $Q$-modules: Suppose $\cH$ and $\cK$ are right $Q$-modules. Then a map $T:\cH\rarrow\cK$ is \textit{conditionally Hilbert--Schmidt} if it is $Q$-linear and $\sum_{\xi\in\Omega} \norm{T\xi}_2^2 < \infty$ for every conditionally orthonormal set $\Omega\subset \cH^0$. However, we will mostly be interested in the specific case of Hilbert--Schmidt operators defined in Definition~\ref{def:HS}.
\end{remark}

\begin{remark}
We point out that the Hilbert space norm $\norm{.}_2$ is used for the sum in the above definition, whereas orthogonality of the involved elements is taken with respect to the conditional inner product $\langle .\,,.\rangle_Q$, whose associated norm is only defined on the dense set of bounded vectors. Also, note that in the unconditional situation where $Q=\C$, we recover the usual notion of a Hilbert--Schmidt operator. 
\end{remark}

In the setting of Definition~\ref{def:HS}, it turns out there is an easy way to describe the conditional Hilbert--Schmidt operators. First, recall that, given a von Neumann algebra inclusion $Q\subset N$, the basic construction $\langle N,e_Q\rangle$ is in general no longer a tracial von Neumann algebra, but it has a canonical semifinite trace $\hat\tau$, which on positive elements takes values in $[0,\infty]$, and which is defined on the dense $^*$-subalgebra $\mathrm{span}\{xe_Q y\mid x,y\in N\}$ by
\[
\hat \tau(xe_Q y) = \tau(xy),
\]
see, e.g., \cite[\S 9]{AP}.

\begin{lemma}\label{lem:HS}
Let $(N,\tau)$ be a tracial von Neumann algebra and $Q\subset N$ a von Neumann subalgebra. An operator $T\in \langle N,e_Q\rangle$ is conditionally Hilbert--Schmidt if and only if $\hat \tau(T^*T)<~\infty$.
\end{lemma}
\begin{proof}
This is immediate from \cite[Proposition 8.4.15]{AP}, which implies that we have
\[
\sum_i \norm{Te_i}_2^2 = \hat\tau(T^*T),
\]
for any operator $T\in \langle N,e_Q\rangle$, and any conditionally orthonormal basis $\{e_i\}_i$ of $L^2(N)$ as a $Q$-module.
\end{proof}

The previous lemma shows that we can view the space of conditional Hilbert--Schmidt operators $HS_Q$ naturally as a subspace of $L^2(\langle N,e_Q\rangle,\hat \tau)$, when we equip $HS_Q$ with its Hilbert--Schmidt norm $\norm{T}_{HS}^2 = \sum_i \norm{Te_i}_2^2 = \hat\tau(T^*T)$. In fact, $L^2(\langle N,e_Q\rangle,\hat \tau)$ is the completion of $HS_Q$ for the Hilbert--Schmidt norm. Note that $HS_Q$ is not necessarily complete: for instance, when $Q=N$, we have $\langle N,e_N\rangle = N$ and $L^2(\langle N,e_N\rangle,\hat \tau) = L^2(N,\tau)$. Nevertheless, the next lemma shows that balls of finite operator-norm radius inside $HS_Q$ are in fact closed, which is analogous to the fact that the unit ball of a tracial von Neumann algebra is complete for the trace norm (which on the unit ball coincides with the strong operator topology).

\begin{lemma}\label{lem:rBallHScomplete}
Let $(N,\tau)$ be a tracial von Neumann algebra and $Q\subset N$ a von Neumann subalgebra. Then for every $r>0$, the convex set
\[
B_r\coloneqq \{T\in HS_Q\mid \norm{T}\leq r\}
\]
is complete in the Hilbert--Schmidt norm. Here $\norm{T}$ denotes the usual operator norm of $T\in B(L^2(N))$.
\end{lemma}
\begin{proof}
Assume $(T_n)_n$ is a Cauchy sequence in $HS_Q$ for the Hilbert--Schmidt norm. 
For $x,y\in HS_Q$, since $yy^*\leq \norm{yy^*}1$, we have $\norm{xy}_{HS}^2 = \hat\tau(xyy^*x^*)\leq \norm{y}^2\norm{x}_{HS}^2$, and thus for any $S\in HS_Q$, we get
\[
\norm{T_nS-T_mS}_{HS} \leq \norm{S}\norm{T_n-T_m}_{HS}.
\]
Since $(T_n)_n$ is Cauchy for the Hilbert--Schmidt norm, we get that $(T_nS)_n$ is Cauchy in $L^2(\langle N,e_Q\rangle,\hat\tau)$ for every $S\in HS_Q$. Since $HS_Q$ is dense inside $L^2(\langle N,e_Q\rangle,\hat\tau)$, we thus get an operator $T\in B(L^2(\langle N,e_Q\rangle,\hat\tau))$ determined by $T(S) = \lim_n T_nS$ for $S\in HS_Q$. In particular, we get that $T_n$ converges in the strong operator topology to $T$, and thus $\norm{T}\leq r$. Since $\langle N,e_Q\rangle\subset B(L^2(\langle N,e_Q\rangle,\hat\tau))$ is a von Neumann algebra, and thus closed in the strong operator topology, we moreover conclude that $T\in \langle N,e_Q\rangle$. Finally, let $(p_i)_i$ be an increasing net of finite projections in $\langle N,e_Q\rangle$ converging to $1$ in the strong operator topology. Then 
\[
\hat\tau(T^*T) = \sup_i \hat\tau(T^*p_iT) = \sup_i \norm{p_iT}_{HS}^2 = \sup_i \lim_n \norm{p_i T_n}_{HS}^2 \leq \sup_i\lim_n \norm{p_i}^2\norm{T_n}_{HS}^2 <\infty.
\]
We conclude that $T\in HS_Q$. Similarly, we get $\hat\tau((T_n-T)^*(T_n-T))\to 0$, and hence $T_n$ converges to $T$ in the Hilbert--Schmidt norm. This finishes the proof of the lemma.
\end{proof}

\subsection{Hilbert--Schmidt operators as convolution operators}\label{ssec:HSconv} In the commutative setting $L^\infty(Y)\subset L^\infty(X)$, given a Hilbert--Schmidt operator $K\in HS_{L^\infty(Y)}\subset \langle L^\infty(X), e_{L^\infty(Y)}\rangle$, we can consider the conditional convolution operator $K\ast_Y \cdot:L^2(X)\rarrow L^2(X)$ given by
\begin{equation}\label{eq:conv}
(K\ast_Y f)(x)\coloneqq \int_X K(x,x') f(x') \, d\mu_{\pi(x)}(x').
\end{equation}
Considering the inclusion $\iota_1: L^2(X)\hookrightarrow L^2(X)\ot_{L^\infty(Y)} L^2(X)$ in the first coordinate, we see that the corresponding orthogonal projection on this subspace is given by
\begin{equation*}
\begin{split}
p_1 = (\id\ot E_{L^\infty(Y)}):L^2(X)\ot_{L^\infty(Y)} L^2(X) \;&\rarrow\; L^2(X)\\
\xi\ot_{L^\infty(Y)}\eta \;&\mapsto\; \xi E_{L^\infty(Y)}(\eta)
\end{split}
\end{equation*}
for $\xi,\eta\in L^\infty(X)$. Similarly for $\iota_2$ and $p_2$. 

Using the canonical isomorphism $L^2(X\times_Y X)\cong L^2(X)\ot_{L^\infty(Y)} L^2(X)$ from \eqref{eq:HSiso}, we then get for $f\in L^2(X)$:
\begin{equation}\label{eq:pointlessconv}
K\ast_Y f = p_1(K\iota_2(f)) = (\id\ot E_{L^\infty(Y)})(K(1\ot_{L^\infty(Y)} f)).
\end{equation}
Indeed, if $K=\xi\ot_{L^\infty(Y)}\eta$ for $\xi,\eta\in L^\infty(X)$ and $f\in L^\infty(X)$, then by definition we have
\begin{equation*}
\begin{split}
((\id\ot E_{L^\infty(Y)})(K(1\ot_{L^\infty(Y)} f)))(x) &= ((\id\ot E_{L^\infty(Y)})(\xi\ot_{L^\infty(Y)} \eta f))(x)\\
&= (\xi E_{L^\infty(Y)}(\eta f))(x)\\
&= \xi(x) \int_X \eta(x')f(x')\,d\mu_{\pi(x)}(x')\\
&= \int_X \xi(x) \eta(x')f(x')\,d\mu_{\pi(x)}(x')\\
\end{split}
\end{equation*}
By linearity and continuity, we then see from \eqref{eq:conv} that \eqref{eq:pointlessconv} indeed holds.

This formalism now readily extends to general von Neumann algebras. Let $Q\subset N$ be an inclusion of tracial von Neumann algebras. Then the Connes fusion tensor product $L^2(N)\ot_Q L^2(N)$ plays the role of the relatively independent product, and it is well-known that $L^2(\langle N,e_Q\rangle, \hat\tau)$ is isomorphic to $L^2(N)\ot_Q L^2(N)$ via the canonical map $xe_Q y\mapsto \hat x\ot_Q \hat y$ (see, e.g., \cite[Exercise 13.13]{AP}).

Now consider a conditional Hilbert--Schmidt operator $K\in HS_Q\subset \langle N,e_Q\rangle$. Viewing $K$ as an operator $K:L^2(N)\to L^2(N)$, we also write $K\ast_Q f$ instead of $K(f)$ for $f\in L^2(N)$ to emphasize the fact that we view $K$ as a convolution operator, and by analogy with the commutative situation. We claim that
\begin{equation}\label{def:conditionalconvolution}
    K\ast_Q f =  (\id\ot E_Q)(K(1\ot_Q f))
\end{equation}
for $f\in L^2(N)$. Indeed, since $\langle N,e_Q\rangle$ is generated as a von Neumann algebra by elements of the form $xe_Qx'$ with $x,x'\in N$, it suffices to check \eqref{def:conditionalconvolution} for operators of that form. Fixing $x,x'\in N$, we then see that for any $y\in N$:
\[
xe_Qx'\ast_Q \hat y = (xe_Qx')(\hat y) = \widehat{xE_Q(x'y)}\in L^2(N).
\]
Similar to the commutative situation, we have that $xe_Qx'$ corresponds to $x\ot_Q x'$ through the aforementioned isomorphism, and thus we see that (dropping the hats for notational convenience)
\[
xe_Q x'\ast_Q y = x E_Q(x'y) = (\id\ot E_Q)(x\ot_Q x'y) = (\id\ot E_Q)((x\ot_Q x')(1\ot_Q y)),
\]
for all $y\in N$. Since $N$ is dense in $L^2(N)$, this proves our claim \eqref{def:conditionalconvolution}. 

We will most commonly use \eqref{def:conditionalconvolution} when working with conditional Hilbert--Schmidt operators in the remainder of this paper, since it is in our situation often computationally easier to work with $L^2(N)\ot_Q L^2(N)$ rather than with $L^2(\langle N,e_Q\rangle, \hat\tau)$.

\subsection{Some finiteness results for conditional Hilbert--Schmidt operators}

In this section, we establish some finiteness results for conditional Hilbert--Schmidt operators, which will help us later in proving that certain associated modules are of finite rank.

We fix a tracial von Neumann algebra $(N,\tau)$, a von Neumann subalgebra $Q\subset N$, and we view $L^2(N)$ as a right $Q$-module. Recall that $L^2(N)^0$ denotes the dense subspace of left bounded vectors equipped with the $Q$-valued inner product $\langle .\,,.\rangle_Q$. On $L^2(N)^0$, we associate with $\langle .\,,.\rangle_Q$ a $Q$-valued norm defined by 
\[
\norm{x}_Q \coloneqq \langle x,x\rangle_Q^{1/2}
\]
We note that in this situation,
\[
\norm{x}_Q^2 = \langle x,x\rangle_Q = E_Q(x^*x)
\]
for $x\in N$. The usual Hilbert space norm on $L^2(N)$ will be denoted by $\norm{.}_2$.

The main result of this section is Lemma~\ref{lem:centertrace} which shows that conditional Hilbert--Schmidt operators satisfy a stronger summability criterion than the one in their definition, where the $2$-norm is replaced by the center-valued trace. For this, we first establish the following fact, which can be viewed as a conditional Bessel inequality. Surely this has been observed by experts before, but we include a proof for the reader's convenience.

\begin{lemma}\label{lem:condBessel}
    If $\Omega\subset L^2(N)^0$ is a conditionally orthonormal set, then for every $K\in HS_Q$,
    \begin{equation}\label{eq:condBessel}
    \sum_{f\in \Omega} (\id\ot E_Q)(K(1\ot_Q f)) (\id\ot E_Q)(K(1\ot_Q f))^* \leq (\id\ot E_Q)(KK^*).
    \end{equation}
\end{lemma}
\begin{proof}
In the computation below, we will write $1\ot f$ instead of $1\ot_Q f$ for notational convenience. If $\Omega$ is finite, we can directly compute:
\begin{align*}
    0&\leq (\id\ot E_Q)\Big[\Big(K^* - \sum_{f\in \Omega} (1\ot f)(\id\ot E_Q)(K(1\ot f))^*\Big)^*\cdot\\
    &\qquad\qquad\qquad\qquad \Big(K^* - \sum_{f\in \Omega} (1\ot f)(\id\ot E_Q)(K(1\ot f))^*\Big)\Big]\\
    &= (\id\ot E_Q)(KK^*) - 2 \sum_{f\in \Omega} (\id\ot E_Q)\left[K(1\ot f)(\id\ot E_Q)(K(1\ot f))^*\right] \\
    %&\qquad - \sum_{f\in \Omega} (\id\ot E_Q)\left[(\id\ot E_Q)(K(1\ot f))(1\ot f)^*K^*\right] \\
    &\qquad + \sum_{f,g\in \Omega} (\id\ot E_Q)\left[(\id\ot E_Q)(K(1\ot f))(1\ot f)^*(1\ot g)(\id \ot E_Q)(K(1\ot g))^*\right]\\
    &= (\id\ot E_Q)(KK^*) - 2\sum_{f\in \Omega} (\id\ot E_Q)(K(1\ot f))(\id\ot E_Q)(K(1\ot f))^* \\
    %&\qquad - \sum_{f\in \Omega} (\id\ot E_Q)(K(1\ot f))(\id\ot E_Q)(K(1\ot f))^* \\
    &\qquad + \sum_{f,g\in \Omega} (\id\ot E_Q)(K(1\ot f)) E_Q(f^*g)(\id \ot E_Q)(K(1\ot g))^*\\
    &\leq (\id\ot E_Q)(KK^*) - \sum_{f\in \Omega} (\id\ot E_Q)(K(1\ot f))(\id\ot E_Q)(K(1\ot f))^*,
\end{align*}
where we used that $\Omega$ is conditionally orthonormal in the last inequality. If $\Omega$ is infinite, it follows from the above that the sum on the left-hand side of \eqref{eq:condBessel} converges, and then the same computation yields the desired result.
\end{proof}

Using this, we can now establish the following stronger summability criterion for conditional Hilbert--Schmidt operators using the $Q$-valued norm, but where for technical reasons we need to pass to adjoints. Below in Lemma~\ref{lem:centertrace}, this technicality disappears thanks to the tracial property of the center-valued trace.

\begin{lemma}\label{lem:Qnorm}
    Let $K\in HS_Q$ be a conditional Hilbert--Schmidt operator. Then there exists a constant $C>0$ such that for any conditionally orthonormal subset $\Omega\subset L^2(N)^0$, we have 
    \[
    \sum_{f\in \Omega} \norm{(K\ast_Q f)^*}_Q^2 \leq C1_Q.
    \]
\end{lemma}
\begin{proof}
Firstly, we note that if $K\in HS_Q\subset \langle N,e_Q\rangle$, then $E_Q(KK^*)\leq \norm{KK^*}1_Q$. Furthermore, through the identification $K\mapsto \hat K\in L^2(N)\ot_Q L^2(N)$, we have $E_Q(KK^*) = (E_Q\ot E_Q)(\hat K\hat K^*)$. We will drop the hat for notational convenience in what follows.

Set $C\coloneqq \norm{KK^*}$. We can now compute:
\begin{align*}
    \sum_{f\in \Omega} \norm{(K\ast_Q f)^*}_Q^2 &= \sum_{f\in \Omega} \langle (K\ast_Q f)^*, (K\ast_Q f)^*\rangle_Q\\
    &= \sum_{f\in \Omega} E_Q\left((\id\ot E_Q)(K(1\ot_Q f)) (\id\ot E_Q)(K(1\ot_Q f))^*\right).
\end{align*}
By Lemma~\ref{lem:condBessel}, we thus get
\[
    \sum_{f\in \Omega} \norm{(K\ast_Q f)^*}_Q^2 \leq E_Q((\id\ot E_Q)(KK^*)) = (E_Q\ot E_Q)(KK^*) \leq C1_Q.
\]
This finishes the proof.
\end{proof}

For the next lemma, we denote by $\cZ = \cZ(Q)$ the center of $Q$, and by $E_\cZ:Q\to\cZ$ the corresponding center-valued trace.

\begin{lemma}\label{lem:centertrace}
    Let $K\in HS_Q$ be a conditional Hilbert--Schmidt operator. Then there exists a constant $C>0$ such that for any conditionally orthonormal subset $\Omega\subset L^2(N)^0$, we have 
    \[
    \sum_{f\in \Omega} E_\cZ(\langle K\ast_Q f, K\ast_Q f\rangle_Q) \leq C1_Q.
    \]
\end{lemma}
\begin{proof}
Using the tracial property of $E_\cZ$, we compute
\begin{align*}
    \sum_{f\in \Omega} E_\cZ(\langle K\ast_Q f, K\ast_Q f\rangle_Q) &= \sum_{f\in \Omega} E_\cZ(E_Q((K\ast_Q f)^* (K\ast_Q f)))\\
    &= \sum_{f\in \Omega} E_\cZ((K\ast_Q f)^* (K\ast_Q f))\\
    &= \sum_{f\in \Omega} E_\cZ((K\ast_Q f) (K\ast_Q f)^*)\\
    &= \sum_{f\in \Omega} E_\cZ(E_Q((K\ast_Q f) (K\ast_Q f)^*)).
\end{align*}
By Lemma~\ref{lem:Qnorm}, $\sum_{f\in \Omega} E_Q((K\ast_Q f) (K\ast_Q f)^*)= \sum_{f\in \Omega} \norm{(K\ast_Q f)^*}_Q^2 \leq C 1_Q$, and hence
\[
\sum_{f\in \Omega} E_\cZ(E_Q((K\ast_Q f) (K\ast_Q f)^*))= E_\cZ\left(\sum_{f\in \Omega} \norm{(K\ast_Q f)^*}_Q^2\right) \leq C 1_Q.
\]
This finishes the proof.
\end{proof}

\section{Compact extensions}\label{sec-compact}

In this section, we define \textit{conditionally (relative to $Q$) almost periodic elements} by analogy with the commutative case, and prove Theorem \ref{thm:compchar}. The following definition is inspired by the operator-theoretic conceptualizations in the commutative case in \cite[\S 2.13]{tao-poincare}, and is similar to the definitions in \cite[\S 3.1, Appendix C]{kerrli}. 

\begin{definition}
Let $Q\subset N$ be an inclusion of von Neumann algebras. Suppose $\cH$ is a right $Q$-module. 
    \begin{itemize}
        \item[(i)] A subset $Z\subset \cH$ is called a \textit{finitely generated module zonotope} if it is of the form $\{\sum_{f\in F} f\cdot x\mid x\in Q, \norm{x}\leq 1\}$ for some finite set $F\subset \cH$.
        \item[(ii)] We call $\cH$ a \textit{normed $Q$-module} if it is equipped with a norm $\norm{.}$ satisfying $\norm{\xi x}\leq \norm{\xi}\norm{x}$ for all $\xi\in\cH$, $x\in Q$.
        \item[(iii)] If $\cH$ is a normed $Q$-module, a set $K\subset \cH$ is called \textit{conditionally precompact} if for every $\eps>0$, there exists a finitely generated module zonotope $Z$ such that $K\subseteq_\eps Z$, that is, $\inf_{z\in Z}\norm{k-z} \leq \eps$ for all $k\in K$.
    \end{itemize}
\end{definition}

We note that working with right modules is just a choice, and one could easily adapt the definitions and results to left modules.

\begin{definition}\label{def:rap}
Let $\Gamma\act (N,\tau)$ be a tracial $W^*$-dynamical system, and let $Q\subset N$ be a $\Gamma$-invariant von Neumann subalgebra. 
Consider $L^2(N)$ as a right $Q$-module in the canonical way, and equip it with its usual 2-norm $\norm{.}_2$ coming from the trace of $N$. This turns $L^2(N)$ into a normed $Q$-module. We call an element $\xi\in L^2(N)$ \textit{almost periodic relative to $Q$} if its orbit is conditionally precompact. If $Q$ is clear from the context, we will also call $\xi$ \textit{conditionally almost periodic}, or just \textit{almost periodic}.
    
    We denote by $\cAP_{Q,N}^2\subset L^2(N)$ the set of all elements in $L^2(N)$ that are almost periodic relative to $Q$, and by $\cAP_{Q,N}= \cAP_{Q,N}^2\cap N$ the set of all elements in $N$ almost periodic relative to $Q$. When $N$ is clear from the context, we also write $\cAP^2_Q$ and $\cAP_Q$.
\end{definition}

Unwrapping the definition, one immediately gets the following.

\begin{lemma}\label{lem:cap}
Let $\Gamma\act^\sigma (N,\tau)$ be a tracial $W^*$-dynamical system, and let $Q\subset N$ be a $\Gamma$-invariant von Neumann subalgebra.  For any $\xi\in L^2(N)$, the following are equivalent.
    \begin{enumerate}
        \item $\xi$ is almost periodic relative to $Q$.
        \item For every $\eps>0$, there exist $\eta_1, \ldots, \eta_k\in L^2(N)$ such that for every $\gamma\in \Gamma$, there exist $\kappa_i(\gamma)\in Q$ with $\sup_{i,\gamma}\norm{\kappa_i(\gamma)}<\infty$ and 
        \[
        \norm{\sigma_\gamma(\xi) - \sum_{i=1}^k \eta_i\kappa_i(\gamma)}_2 \leq \eps.
        \]
    \end{enumerate}
\end{lemma}

Note that we could also pick the $\eta_i$ in item (2) above from $N$ instead of $L^2(N)$ by density. The following properties follow from a straightforward computation which is left to the interested reader.

\begin{lemma}\label{lem:closed}
The following statements hold. 
    \begin{itemize}
        \item[(i)] $\cAP_{Q,N}^2\subset L^2(N)$ is a $\Gamma$-invariant Hilbert subspace.
        \item[(ii)] $\cAP_{Q,N}\subset N$ is a $\Gamma$-invariant linear subspace which is closed in the strong operator topology.
    \end{itemize}
\end{lemma}

We can now use the above to formally define compact extensions in the von Neumann algebra setting.

\begin{definition}\label{def:relcompact}
(cf. \cite[Definition 3.9]{chifan-das}) Let $\Gamma\act (N,\tau)$ be a tracial $W^*$-dynamical system, and let $Q\subset N$ be a $\Gamma$-invariant von Neumann subalgebra.  We call $N$ a \textit{compact extension} of $Q$ if $\cAP_{Q,N}^2 = L^2(N)$ (or equivalently $\cAP_{Q,N} = N$).
\end{definition}

Before proving Theorem~\ref{thm:compchar}, we make the following observation:

\begin{remark}
    In the setting of Definition~\ref{def:relcompact}, given a conditional Hilbert--Schmidt operator $K\in HS_Q$, we can view $K$ as an operator $K:L^2(N)\to L^2(N)$, or as an element of $L^2(N)\ot_Q L^2(N)$ (see Subsection~\ref{ssec:HSconv}). It is then a straightforward exercise to see that $K:L^2(N)\to L^2(N)$ is $\Gamma$-equivariant if and only if $K\in L^2(N)\ot_Q L^2(N)$ is $\Gamma$-invariant. We will freely make use of this observation without explicit mentioning.
\end{remark}

We can now give the proof of our first main theorem, which we recall here for convenience. 

\compchar*

\begin{proof}
$(1)\Rightarrow (2)$. Assume (1) holds and suppose that $f\in N$ is orthogonal to every function of the form $K\ast_Q \xi$ with $K\in HS_Q$ $\Gamma$-equivariant and $\xi\in L^2(N)$. It suffices to prove that $f=0$.

Consider $f\ot_Q f^*\in L^2(N)\ot_Q L^2(N)$ and let $K\in L^2(N)\ot_Q L^2(N)$ be the unique element of minimal norm in the closed convex hull $\cC$ of $\Orb(f\ot_Q f^*)=\{(\sigma_\gamma\times\sigma_\gamma)(f\ot_Q f^*)\mid \gamma\in\Gamma\}$. Since $f\ot_Q f^* \in HS_Q\subset L^2(N)\ot_Q L^2(N)$, we get that the convex hull of $\Orb(f\ot_Q f^*)$ is contained in $HS_Q$. Moreover, this convex hull is bounded in norm by $\norm{f}^2$. By Lemma~\ref{lem:rBallHScomplete}, it follows that $\cC\subset HS_Q$, and in particular, $K\in HS_Q$. By construction, $K$ is moreover $\Gamma$-invariant, and thus by assumption, $f$ is orthogonal to $K\ast_Q f$. In other words, we have
\begin{align*}
0 = \langle f, K\ast_Q f\rangle_{L^2(N)} &= \tau(f^* (\id\ot E_Q)(K(1\ot_Q f)))\\
&= \tau((\id\ot E_Q)(K(f^*\ot_Q f)))\\
&= \langle K, f\ot_Q f^*\rangle_{L^2(N)\ot_Q L^2(N)}.
\end{align*}
We thus get that $f\ot_Q f^*$ is orthogonal to $K$, and since $K$ is $\Gamma$-invariant, we conclude that $(\sigma_\gamma\times\sigma_\gamma)(f\ot_Q f^*)$ is orthogonal to $K$ for every $\gamma\in\Gamma$. Hence $K$ is orthogonal to itself and therefore $K=0$.

Fix $\eps>0$. Following Lemma \ref{lem:cap}(2), take a finite set $\eta_1, \dots, \eta_k \in N$ and for every $\gamma\in\Gamma$, $\kappa_i(\gamma)\in Q$, such that $\sup_{i,\gamma}\norm{\kappa_i(\gamma)} < \infty$ and 
\[
\norm{\sigma_\gamma(f) -  \sum_{i=1}^k \eta_i\kappa_i(\gamma)}_2 \leq \sqrt{\eps}.
\]
By the previous paragraph, we can find a sequence $(K_n)_n$ in the convex hull of $\Orb(f\ot_Q f^*)$ such that $\norm{K_n}_{L^2(N)\ot_Q L^2(N)}\rarrow 0$. In particular, for every $1\leq i\leq k$, we have $\langle K_n, \eta_i\ot_Q \eta_i^*\rangle_{L^2(N)\ot_Q L^2(N)}\rarrow 0$, and thus
\[
\sum_{i=1}^k \langle K_n, \eta_i\ot_Q \eta_i^*\rangle_{L^2(N)\ot_Q L^2(N)}\rarrow 0.
\]
For a given $n$, writing $K_n$ as a convex combination
\[
K_n = \sum_{j=1}^{m_n} \lambda_j (\sigma_{\gamma_j}\times\sigma_{\gamma_j})(f\ot_Q f^*),
\]
we can expand the inner product to get
\begin{align*}
\sum_{i=1}^k\langle K_n, \eta_i\ot_Q \eta_i^*\rangle_{L^2(N)\ot_Q L^2(N)} &= \sum_{i=1}^k\langle \sum_{j=1}^{m_n} \lambda_j (\sigma_{\gamma_j}\times\sigma_{\gamma_j})(f\ot_Q f^*), \eta_i\ot_Q \eta_i^*\rangle_{L^2(N)\ot_Q L^2(N)}\\
&= \sum_{i=1}^k\sum_{j=1}^{m_n} \lambda_j\langle \sigma_{\gamma_j}(f)^*, \langle \sigma_{\gamma_j}(f), \eta_i\rangle_Q \eta_i^*\rangle_{L^2(N)}\\
&= \sum_{i=1}^k\sum_{j=1}^{m_n} \lambda_j \tau(\sigma_{\gamma_j}(f) E_Q(\sigma_{\gamma_j}(f)^*\eta_i)\eta_i^*)\\
&= \sum_{i=1}^k\sum_{j=1}^{m_n} \lambda_j \tau(E_Q(\sigma_{\gamma_j}(f)^*\eta_i)E_Q(\eta_i^*\sigma_{\gamma_j}(f)))\\
&= \sum_{j=1}^{m_n} \lambda_j \left(\sum_{i=1}^k\norm{E_Q(\sigma_{\gamma_j}(f)^*\eta_i)}_2^2\right).
\end{align*}
Since the latter convex combination converges to zero, we conclude that upon taking $n$ large enough, there exists $\gamma\in\Gamma$ such that for every $1\leq i\leq k$,
\begin{equation}\label{eq:fg}
\norm{E_Q(\sigma_{\gamma}(f)^*\eta_i)}_2 < \frac{\eps}{k\cdot \sup_{i,\gamma}\norm{\kappa_i(\gamma)}}.
\end{equation}
We can then calculate 
\begin{align*}
\norm{\sigma_\gamma(f)}_2^2 &= \tau(\sigma_\gamma(f)^*\sigma_\gamma(f))\\ 
&= \tau(E_Q(\sigma_\gamma(f)^*\sigma_\gamma(f)))\\
&\leq \tau\left(E_Q(\sigma_\gamma(f)^*\sigma_\gamma(f)) + E_Q\left(\left(\sum_{i=1}^k \eta_i\kappa_i(\gamma)\right)^{\!\! *}\left(\sum_{i=1}^k \eta_i\kappa_i(\gamma)\right)\right)\right)\\
&= \tau\bigg(E_Q\left(\left(\sigma_\gamma(f)-\sum_{i=1}^k \eta_i\kappa_i(\gamma)\right)^{\!\! *}\left(\sigma_\gamma(f)-\sum_{i=1}^k \eta_i\kappa_i(\gamma)\right)\right) + \\
&\qquad + E_Q\left(\sigma_\gamma(f)^*\left(\sum_{i=1}^k \eta_i\kappa_i(\gamma)\right)\right) + E_Q\left(\left(\sum_{i=1}^k \eta_i\kappa_i(\gamma)\right)^{\!\! *}\sigma_\gamma(f)\right)\bigg)\\
&\leq \norm{\sigma_\gamma(f) -  \sum_{i=1}^k \eta_i\kappa_i(\gamma)}_2^2 + 2\sum_{i=1}^k \norm{E_Q(\sigma_{\gamma}(f)^*\eta_i)}_2\norm{\kappa_i(\gamma)}\\
&< 3\eps.
\end{align*}
Since $\eps$ was arbitrary, we conclude that $f=0$.

$(2)\Rightarrow (3)$. 
Let $K\in HS_Q$ be $\Gamma$-equivariant and assume w.l.o.g.\ that $K\geq 0$. Write $p_\eps\coloneqq 1_{[\eps, \infty)}(K)$, where we apply Borel functional calculus to $K$ as an operator $K:L^2(N)\rarrow L^2(N)$. Then $p_\eps$ is $\Gamma$-invariant, since it arises as a limit of polynomials in $K$, and thus $\cH_\eps\coloneqq p_\eps L^2(N)$ is a $\Gamma$-invariant $Q$-submodule of $L^2(N)$. 

Moreover, by construction we have $p_\eps K = K p_\eps \geq \eps p_\eps$, or in other words $K\geq \eps 1$ on the submodule $\cH_\eps$. In particular, 
\begin{equation}\label{eq:L2ineq}
\langle K\ast_Q f,f\rangle_{L^2(N)} \geq \eps \langle f,f\rangle_{L^2(N)}
\end{equation}
for every $f\in \cH_\eps$. We claim that moreover
\begin{equation}\label{eq:EZineq}
\abs{E_\cZ(\langle K\ast_Q f,f\rangle_{Q})} \geq \eps E_\cZ(\langle f,f\rangle_{Q}).
\end{equation}
Indeed, the latter inequality can be interpreted as an inequality of functions upon identifying $\cZ\cong L^\infty(X,\mu)$ for some probability space $(X,\mu)$. Now if \eqref{eq:EZineq} fails, then we can find a positive measure subset $A\subset X$ such that, denoting $q = 1_A\in \cZ$,
\[
\abs{E_\cZ(\langle K\ast_Q fq,fq\rangle_{Q})} = \abs{E_\cZ(\langle K\ast_Q f,f\rangle_{Q})}q < \eps E_\cZ(\langle f,f\rangle_{Q})q = \eps E_\cZ(\langle fq,fq\rangle_{Q}).
\]
However, this would imply that $\langle K\ast_Q fq,fq\rangle_{L^2(N)} = \tau(\langle K\ast_Q fq,fq\rangle_{Q}) < \eps \langle fq,fq\rangle_{L^2(N)}$, contradicting \eqref{eq:L2ineq} since $fq\in \cH_\eps$.

Viewing $\cH_\eps$ as a right $\cZ$-module, we note that
\[
\langle f,g\rangle_\cZ = E_\cZ(f^*g) = E_\cZ(E_Q(f^*g)) = E_\cZ(\langle f,g\rangle_Q),
\]
for $f,g\in \cH_\eps$. Since $\cZ$ is moreover abelian, we have a conditional Cauchy-Schwarz inequality: for all $f,g\in \cH_\eps$,
\[
\abs{\langle f,g\rangle_\cZ}^2 \leq \langle f,f\rangle_\cZ \langle g,g\rangle_\cZ,
\]
see \cite{Ka53}. 

Using this together with \eqref{eq:EZineq}, we thus get
\begin{align*}
     \eps E_\cZ(\langle f,f\rangle_Q) &\leq \abs{E_\cZ(\langle K\ast_Q f,f\rangle_{Q})}\\
    &= \abs{\langle K\ast_Q f,f\rangle_{\cZ}}\\
    &\leq \langle K\ast_Q f, K\ast_Q f\rangle_{\cZ}^{\frac{1}{2}} \langle f,f\rangle_{\cZ}^{\frac{1}{2}} \\
    &= E_\cZ(\langle K\ast_Q f,K\ast_Q f\rangle_{Q})^{\frac{1}{2}} E_\cZ(\langle f,f\rangle_{Q})^{\frac{1}{2}}
\end{align*}
Letting $\Omega\subset \cH_\eps$ be any conditionally orthonormal basis, which exists by \cite[Proposition~8.4.11]{AP}, we thus get that
\begin{align*}
    \sum_{f\in \Omega} E_\cZ(\langle f,f\rangle_Q) \leq \frac{1}{\eps^2}\sum_{f\in \Omega} E_\cZ(\langle K\ast_Q f,K\ast_Q f\rangle_{Q}).
\end{align*}
Since $K\in HS_Q$, we conclude from Lemma~\ref{lem:centertrace} that the right-hand side is bounded above by $\frac{C}{\eps^2}1_Q$ for some constant $C>0$. Proposition~\ref{prop:finiterank} thus implies that $\cH_\eps$ is finite rank.

Letting $\eps\rarrow 0$, we get that the union of the finitely generated $\Gamma$-invariant $Q$-submodules of the range of $K$ is dense in the range of $K$. This concludes the proof of $(2)\Rightarrow (3)$.

$(3)\Rightarrow (1)$. Let $\cH$ be a closed $\Gamma$-invariant $Q$-submodule of $L^2(N)$, which is generated as a $Q$-module by a finite orthonormal set of vectors $F\subset \cH$. Take $f\in \cH$. Then the orbit of $f$ is contained in $\sum_{h\in F} \norm{f}h B_Q$, where $B_Q$ denotes the unit ball of $Q$, which is by definition a $Q$-finitely generated module zonotope in $L^2(N)$. By assumption such modules generate $L^2(N)$, and thus we get that the space of conditionally almost periodic elements is dense in $L^2(N)$. Since this space is also closed by Lemma \ref{lem:closed}, it follows that (1) holds.
\end{proof}

\begin{remark}\label{rem-duvenhage}
Let $\Gamma\act (N,\tau)$ be a $W^*$-dynamical system, and assume $Q\subset P\subset N$ are $\Gamma$-invariant von Neumann subalgebras. We note that it follows immediately from Definition \ref{def:relcompact} that if $\Gamma\act (Q\subset N)$ is a compact extension, then also $\Gamma\act (Q\subset P)$ and $\Gamma\act (P\subset N)$ are compact extensions. Together with Theorem \ref{thm:compchar}(3), this answers a question raised in \cite[Section 6]{duvenhage-rel-compact}. Moreover, one can view relatively almost periodic functions as defined in Definition \ref{def:rap} as ``generalized eigenfunctions'' (cf. \cite[Definition 9.10]{glasner}). We thus show that these generalized eigenfunctions arise from conditional Hilbert--Schmidt operators, and through the equivalence with the characterization in terms of finite rank modules in Theorem \ref{thm:compchar}(3), this answers the other question raised in \cite[Section 6]{duvenhage-rel-compact} as well.
\end{remark}

As another corollary of the proof of Theorem \ref{thm:compchar}, we can answer a question from Austin, Eisner and Tao \cite[Question 4.3]{austin}. 
Consider again the basic construction $\langle N,e_Q\rangle$, and recall that $\langle N,e_Q\rangle$ coincides with $Q_{\text{right}}'\cap B(L^2(N))$, where $Q_{\text{right}}$ denotes right multiplication of $Q$. In terms of modules, this implies that if $V\subset L^2(N)$ is a closed right $Q$-submodule, then the orthogonal projection $P_V:L^2(N)\rarrow V$ belongs to $\langle N,e_Q\rangle$. One can use this to define the \textit{$Q$-dimension of V}, by letting
\[
\dim_Q(V)\coloneqq \hat\tau(P_V).
\]
We say that $V$ has finite lifted trace if $\hat\tau(P_V)<\infty$. We note that in general, having finite lifted trace is not equivalent to being finitely generated, e.g., see \cite[Section 8]{AP}. 

Given a tracial $W^*$-dynamical system $\Gamma\act (N,\tau)$, and a $\Gamma$-invariant von Neumann subalgebra $Q\subset N$, \cite[Question 4.3]{austin} asks whether any $\Gamma$-invariant right $Q$-submodule $V\subset L^2(N)$ of finite lifted trace can be approximated by $\Gamma$-invariant finitely generated right $Q$-submodules of $L^2(N)$. \cite[Lemma 4.1]{austin} gives a partial answer, by establishing this fact for $\Z$-actions and under the additional (restrictive) condition that $Q\subset N$ is central. We next observe that the proof of Theorem \ref{thm:compchar} can be used to answer their question in full generality. 

\begin{lemma}\label{lem-aet}
Let $\Gamma\act (N,\tau)$ be a tracial $W^*$-dynamical system, and let $Q\subset N$ be a $\Gamma$-invariant von Neumann subalgebra.  Suppose $V\subset L^2(N)$ is a $\Gamma$-invariant right $Q$-submodule of finite lifted trace. Then for any $\eps>0$, there exists a further $\Gamma$-invariant right $Q$-submodule $V_1\subset V$ such that
\begin{itemize}[nolistsep]
    \item $\hat\tau(P_V-P_{V_1})<\eps$, and
    \item $V_1$ is finitely generated.
\end{itemize}
\end{lemma}
\begin{proof}
Since $P_V\in \langle N,e_Q\rangle$, it follows from Lemma~\ref{lem:HS} that the condition that $V\subset L^2(N)$ has finite lifted trace exactly means that $P_V$ is a conditional Hilbert--Schmidt operator. Therefore, if $V\subset L^2(N)$ is a $\Gamma$-invariant right $Q$-submodule of finite lifted trace, then $P_V$ is a $\Gamma$-equivariant conditional Hilbert--Schmidt operator. The proof of $(2)\Rightarrow (3)$ in Theorem \ref{thm:compchar} now immediately implies the result.
\end{proof}

We saw in Lemma \ref{lem:closed} that the set $\cAP^2_{Q,N}$ of conditionally almost periodic elements in $L^2(N)$ is always a closed Hilbert subspace. However its intersection $\cAP_{Q,N}$ with $N$ is in general only an SOT-closed linear subspace of $N$ that may not be closed under multiplication and/or involution. We describe below an example of this phenomenon due to Austin, Eisner and Tao (see \cite[Example 4.2]{austin}). If this happens, one would not be able to reasonably consider $\cAP_{Q,N}$ as the maximal compact extension of $Q$ inside $N$ within the framework of actions on von Neumann algebras. 

\begin{example}\label{ex:AET}
Consider $Q\coloneqq L(\Z)=L(\langle a\rangle)\subset L(\langle a,b\rangle) = L(\bF_2) \eqqcolon N$. We consider the action of $\Z$ on $N$ determined by the automorphism $\alpha$ of $\bF_2$ given by $\alpha(a) = a$ and $\alpha(b) = ba$, which obviously leaves $Q$ invariant. It is then easy to see that $b$ is almost periodic relative to $Q$, as it is contained in the $\Z$-invariant $Q$-module $V$ of rank 1 given by $V\coloneqq \overline{\mathrm{span}}\{\widehat{ba^n}\mid n\in \Z\}=bL^2(Q)\subset L^2(N)$. However, $b^2$ is not almost periodic relative to $Q$, since the elements $\alpha^n(b^2) = ba^nba^n$ for $n\in \Z$ are pairwise orthogonal relative to $Q$. 
    
Furthermore, one can observe that it is easy to get an ergodic such example, i.e., an action $\Gamma\act N$ such that the space of fixed points $N^\Gamma$ is trivial ($N^\Gamma = \C 1$). Indeed, one can change the above action on $Q=L(\Z)\cong L^\infty(\bT)$ to be associated to any ergodic action on $\bT$, such as irrational rotation. We would still have that $b$ is contained in $V$ as defined above, and furthermore $\alpha^n(b^2) = ba\alpha^{n-1}(a)ba\alpha^{n-1}(a)$, which is still a sequence of pairwise orthogonal elements relative to $Q$. 
\end{example}

In a preprint of Chifan and Peterson \cite{chifan-peterson}, it is observed that $\cAP_{Q,N}$ is a von Neumann subalgebra under certain additional assumptions. For this, we first introduce the following definition: We call a von Neumann subalgebra $Q\subset N$ \textit{quasi-regular} if the quasi-normalizer $\cQN_N(Q)\coloneqq\{x\in N\mid \exists x_1, \ldots, x_k\in N: xQ\subset \sum_{i=1}^k Qx_i \text{ and } Qx\subset \sum_{i=1}^k x_iQ\}$ generates $N$ as a von Neumann algebra.

Unwrapping the definitions, it is not too hard to see that under the additional assumption that $Q$ is quasi-regular in $N$, we do always get that the conditionally almost periodic elements form a von Neumann subalgebra. This was first observed by Chifan and Peterson in \cite{chifan-peterson}, but we provide a proof for completeness.

\begin{proposition}\label{prop:quasireg}\cite{chifan-peterson}
Let $\Gamma\act (N,\tau)$ be a tracial $W^*$-dynamical system, and let $Q\subset N$ be a quasi-regular and $\Gamma$-invariant von Neumann subalgebra. Then $\cAP_{Q,N}$ forms a von Neumann subalgebra of $N$.
\end{proposition}
\begin{proof}
We already know from Lemma~\ref{lem:closed} that $\cAP\coloneqq \cAP_{Q,N}$ is a linear subspace which is closed in the strong operator topology. We thus only have to argue that it is closed under multiplication and convolution. 

Let $\xi_1,\xi_2\in \cAP$ and fix $\eps>0$. By definition, for $t=1,2$ we can find $\eta_1^{t}, \ldots, \eta_k^{t}\in N$ such that for every $\gamma\in \Gamma$, there exists $\kappa_i^{t}(\gamma)\in Q$ such that
\begin{align*}
\sup_{i,\gamma,t}\norm{\kappa_i^{t}(\gamma)}_\infty&\leq M<\infty,\\
\norm{\sigma_\gamma(\xi_1) - \sum_{i=1}^k \eta_i^{1}\kappa_i^{1}(\gamma)}_2 &\leq \frac{\eps}{4\norm{\xi_2}_\infty}, \text{ and}\\
\norm{\sigma_\gamma(\xi_2) - \sum_{i=1}^k \eta_i^{2}\kappa_i^{2}(\gamma)}_2 &\leq \frac{\eps}{4M\sum_i\norm{\eta_i^1}_\infty}.
\end{align*}
In particular, we have
\[
\norm{\sigma_\gamma(\xi_1\xi_2) - \sum_{i,j=1}^k \eta_i^{1}\kappa_i^{1}(\gamma)\eta_j^{2}\kappa_j^{2}(\gamma)}_2 \leq \frac{\eps}{2}.
\]
Since $\cQN_N(Q)''=N$, we can moreover assume that $\eta_i^{t}\in \cQN_N(Q)$.

Denote by $V\subset L^2(N)$ the $Q$-bimodule generated by $\eta_1^{2}, \ldots, \eta_k^{2}$, and by $P_V:L^2(N)\rightarrow V$ the orthogonal projection onto $V$. Using \cite[Lemma~3.5]{FGS} (cf. \cite[Theorem~1.4.2]{popa-betti}), we can find a sequence of projections $z_n\in Q'\cap N$ such that $z_n\rightarrow 1$ in the strong operator topology, together with finitely many mutually $Q$-orthogonal elements $x_{n,1}, \ldots, x_{n,l}\in N$ for every $n$, such that for every $\eta\in N$ we have
\begin{equation}\label{eq:E_Q}
z_nP_Vz_n(\eta) = \sum_m x_{n,m}E_Q(x_{n,m}^*\eta).
\end{equation}
Since $P_V-z_nP_Vz_n\rightarrow 0$ in the strong operator topology as $n\rightarrow\infty$, there exists $n$ large enough such that for all $j$,
\[
\norm{P_V(\eta_j^{2}) - z_nP_Vz_n(\eta_j^{2})}_2\leq \frac{\eps}{2 M^2 k^2 \max_i \norm{\eta_i^1}_\infty}.
\] 
Fix such $n$ and denote $x_m\coloneqq x_{n,m}$. Since $V$ is a $Q$-bimodule and $z_n\in Q'\cap N$, we note that for every $a,b\in Q$ and $\eta\in N$ we have
\[
P_V(a\eta b)=aP_V(\eta)b, \quad\text{and}\quad z_nP_Vz_n(a\eta b) = az_nP_Vz_n(\eta)b.
\]
Hence, using \eqref{eq:E_Q}, we get
\begin{align*}
&\norm{\sum_{i,j} \kappa_i^1(\gamma)\eta_j^{2} \kappa_j^2(\gamma)- \sum_{i,j,m} x_m E_Q(x_m^* \kappa_i^1(\gamma) \eta_j^{2} \kappa_j^2(\gamma))}_2  \\
&\qquad\qquad=\norm{\sum_{i,j} \kappa_i^1(\gamma)\eta_j^{2}\kappa_j^2(\gamma) - z_nP_Vz_n\big(\sum_{i,j} \kappa_i^1(\gamma)\eta_j^{2}\kappa_j^2(\gamma)\big)}_2\\
&\qquad\qquad\leq \sum_{i,j} \norm{P_V(\kappa_i^1(\gamma)\eta_j^{2}\kappa_j^2(\gamma)) - z_nP_Vz_n(\kappa_i^1(\gamma)\eta_j^{2}\kappa_j^2(\gamma))}_2\\ 
&\qquad\qquad\leq \sum_{i,j} M^2\norm{P_V(\eta_j^{2}) - z_nP_Vz_n(\eta_j^{2})}_2\\
&\qquad\qquad\leq \frac{\eps}{2\max_i \norm{\eta_i^1}_\infty}.
\end{align*}
Denoting $\kappa_{i,m}(\gamma)\coloneqq \sum_{j} E_Q(x_m^* \kappa_i^1(\gamma) \eta_j^{2} \kappa_j^2(\gamma))\in Q$, it is then a straightforward calculation to check that 
\[
\sup_{i,m,\gamma}\norm{\kappa_{i,m}(\gamma)}_\infty \leq M^2 \max_{m} \norm{x_m}_\infty \max_j \norm{\eta_j^{2}}_\infty<\infty,
\]
and
\begin{align*}
&\norm{\sigma_\gamma(\xi_1\xi_2) - \sum_{i,m} (\eta_i^1 x_m)\kappa_{i,m}(\gamma)}_2 \\
&\qquad\leq \frac{\eps}{2} + \norm{\sum_i \eta_i^1\Big(\sum_j \kappa_i^{1}(\gamma)\eta_j^{2}\kappa_j^{2}(\gamma) - \sum_{j,m} x_m E_Q(x_m^* \kappa_i^1(\gamma) \eta_j^{2} \kappa_j^2(\gamma))\Big)}_2\\
&\qquad\leq \frac{\eps}{2} + \max_i \norm{\eta_i^1}_\infty\frac{\eps}{2 \max_i \norm{\eta_i^1}_\infty}\\
&\qquad= \eps.
\end{align*}
Hence $\{\eta_i^1x_m\}_{i,m}$ and $\kappa_{i,m}(\gamma)$ witness that $\xi_1\xi_2\in \cAP$. The proof that $\cAP$ is closed under convolution proceeds similarly and is left to the reader.
\end{proof}

This now leads to the following definition. 

\begin{definition}
Let $\Gamma\act (N,\tau)$ be a tracial $W^*$-dynamical system, and let $Q\subset N$ be a $\Gamma$-invariant von Neumann subalgebra. Assume $Q\subset N$ is quasi-regular. Then the von Neumann subalgebra $\cAP_{Q,N}\subset N$ consisting of all almost periodic elements relative to $Q$ is called the \textit{maximal compact extension} of $Q$ inside $N$.
\end{definition}

\begin{remark}\label{rk:MaxCpt}
	It follows easily from the proof of Theorem \ref{thm:compchar}, that if the space $\cAP_Q$ of almost periodic elements relative to $Q$ forms a von Neumann subalgebra, it can equivalently be described as 
	\[
	\cAP_Q = \text{closure of } \mathrm{span}\{K\ast_Q x\mid K\in HS_Q\; \Gamma\text{-equivariant}, x\in N\}.
	\]
\end{remark}

\begin{example}
    Let $(N,\tau,\sigma)$ be a $W^*$-dynamical $\Gamma$-system. Then we can consider the ``maximal compact part'' of the action, given by $\cAP_\C\subset N$.
\end{example}

\begin{remark}\label{ex:qreg}
Example \ref{ex:AET} does not only give an example of an action where the almost periodic elements don't form a von Neumann subalgebra, it also shows that even if $Q\subset N$ is quasi-regular, it is not necessarily true that also $\cAP_{Q}\subset N$ is quasi-regular again, nor that $\cAP_{\cAP_{Q}}\subset N$ is a von Neumann subalgebra. Indeed, in the aforementioned example, we see that $Q=\cAP_{\C}$, but $\cAP_Q\subset N$ is not closed under multiplication.
\end{remark}

\section{Joinings}\label{sec:joingings}

Analogous to classical ergodic theory (e.g., see \cite[\S 6]{glasner}), one can consider \textit{joinings} of two $W^*$-dynamical systems. This notion is also considered for general $W^*$-dynamical systems in \cite{joining-1}. In the tracial setting, the definition can be phrased as follows.

\begin{definition}\label{def:joining}
    Let $\Gamma \act (M,\tau_M)$ and $\Gamma\act (N,\tau_N)$ be tracial $W^*$-dynamical systems, and suppose $Q\subset M,N$ is a common $\Gamma$-invariant von Neumann subalgebra. A \textit{joining} of $M$ and $N$ over $Q$ is a $\Gamma$-invariant state $\varphi$ on the algebraic tensor product $M\algot N^{\op}$ such that $\varphi|_{M\ot 1} = \tau_M$, $\varphi|_{1\ot N^{\op}} = \tau_{N}$ and $\varphi|_{Q\algot Q^{\op}} = \Delta_Q$, where $\Delta_Q$ is the ``diagonal'' state on $Q\algot Q^{\op}$ given by $\Delta_Q(a\ot b)=\tau_Q(ab)$.
\end{definition}

\begin{remark}
Note that we can only define the joining $\varphi$ on the algebraic tensor product $M\algot N^{\op}$, since the diagonal state $\Delta_Q$ is in general not bounded and not normal, and thus doesn't necessarily extend to the von Neumann algebraic tensor product $Q\otb Q$. Indeed, if $Q$ is for instance a separable II$_1$ factor, then by Connes' theorem \cite[Theorem~5.1]{connes-hyperfinite}, $\Delta_Q$ is not min-continuous on $Q\algot Q$ unless $Q$ is hyperfinite. We would like to thank Cyril Houdayer for kindly pointing out to us that the definition as it appeared in the first arXiv-version was incorrect. However, we note that all results involving joinings from that first version remain true with the above corrected definition (where we modified the proofs accordingly).
\end{remark}

\begin{remark}
    Through the GNS-construction, one can also formulate joinings in terms of pointed bimodules and therefore also ucp maps. We refer to Lemma~\ref{lem:disjoint}, where this is needed and worked out, and to \cite{joining-1} for further results and details.
\end{remark}

\begin{example}
    Given $M,N,Q$ as in Definition \ref{def:joining}, there always exists a canonical joining $\varphi$ of $M$ and $N$ over $Q$ defined by $\varphi(x\ot y) = \Delta_Q(E_Q(x)\ot E_Q(y)) = \tau_Q(E_Q(x)E_Q(y))$. 
    We call this joining the \textit{relatively independent joining of $M$ and $N$ over $Q$} and denote it by $\tau_M\ot_Q\tau_N$. 
\end{example}

We will often want to work with the GNS construction (i.e. separation-completion) of the algebraic tensor product $M\algot N^\op$ associated with $\tau_M\ot_Q\tau_N$, so we fix its notation now.

\begin{definition}
Given $M,N,Q$ as in Definition \ref{def:joining}, we denote by $L^2(M\algot N^\op,\tau_M\ot_Q\tau_N)$ the separation-completion of $M\algot N^\op$ with respect to the relatively independent joining $\tau_M\ot_Q\tau_N$. 
\end{definition}

We observe next that this GNS representation can be identified with the usual Connes fusion tensor product of the standard representations $L^2(M)$ and $L^2(N)$ over $Q$. Nevertheless, it will sometimes be useful to take the joining point of view.

\begin{lemma}\label{lem:joinConnes}
Given $M,N,Q$ as in Definition \ref{def:joining}, we have a canonical isomorphism of Hilbert spaces
\[
L^2(M\algot N^\op,\tau_M\ot_Q\tau_N) \cong L^2(M,\tau_M)\ot_Q L^2(N^\op,\tau_N).
\]
\end{lemma}
\begin{proof}
This is a straightforward computation and is left to the interested reader.
\end{proof}

The following proposition is an analogue of a well-known important result in Furstenberg--Zimmer structure theory (see, e.g., \cite{furstenberg-book}) in the von Neumann algebraic setting, and it will make use crucially of the identifications in Theorem~\ref{thm:compchar}. It will for instance be useful later when discussing the Host--Kra--Ziegler tower of compact extensions coming from cubic systems. We state the result for a joining of a tracial $W^*$-dynamical system with itself, since that is the setting in which we will use it and in which we discussed conditional Hilbert--Schmidt operators. 

\begin{proposition}\label{prop:invJoin}
Let $\Gamma\act (N,\tau_N)$ be a tracial $W^*$-dynamical system, and suppose $Q\subset N$ is a $\Gamma$-invariant von Neumann subalgebra. If $\xi\in HS_Q\subset L^2(N)\ot_Q L^2(N^\op)$ is $\Gamma$-invariant, then $\xi\in \cAP^2_{Q,N}\ot_Q \cAP^2_{Q,N}$.
\end{proposition}
\begin{proof}
Assume w.l.o.g. that $\xi\geq 0$ as an operator on $L^2(N)$. As in the proof of $(2)\Rightarrow (3)$ in Theorem~\ref{thm:compchar} we see that the range of $p_\eps\coloneqq 1_{[\eps,\infty)}(\xi):L^2(N)\rarrow L^2(N)$ is a $Q$-finitely generated submodule of $L^2(N)$. This easily implies (by Theorem~\ref{thm:compchar}(3)) that $p_\eps \xi \in \cAP^2_{Q,M}\ot_Q \cAP^2_{Q,N}$. Since $p_\eps \xi\rarrow \xi$ as $\eps\rarrow 0$, we get that also $\xi\in \cAP^2_{Q,M}\ot_Q \cAP^2_{Q,N}$. 
\end{proof}

\begin{corollary}\label{cor:invariantvectors}
    Let $\Gamma\act (N,\tau_N)$ be a tracial $W^*$-dynamical system. Denote by $I\subset N\otb N$ the subalgebra of invariant vectors for the diagonal product action of $\Gamma$ on $N\otb N$. Then $I\subset \cAP_N\otb \cAP_N$.
\end{corollary}

In the commutative setting, the identification in Lemma \ref{lem:joinConnes} can be carried over to the measure space level\footnote{Similar constructions exist for non-standard spaces and in the category of probability algebras by working with canonical models and disintegrations, e.g., see \cite{jt-foundational}.}, by observing that $X_1\times X_2$ equipped with the relatively independent joining measure $\mu_1\times_Y\mu_2$ is isomorphic as a measure algebra to $(X_1\times_Y X_2,\mu_1\times_Y \mu_2)$ as defined in Subsection \ref{ssec:relprod}, where $(Y,\nu)$ is a common factor of $(X_1,\mu_1)$ and $(X_2,\mu_2)$. In the von Neumann algebraic setting however, the ``relatively independent product'' $N_1\ot_Q N_2$ of $Q\subset N_1$ and $Q\subset N_2$ is (usually) not well-defined as a tracial von Neumann algebra if approached the same way. Given a tracial von Neumann algebra $Q$, a right $Q$-module $\cH$ and a left $Q$-module $\cK$, the Connes fusion tensor product $\cH\ot_Q \cK$ is a Hilbert space, but not necessarily a $Q$-module. Indeed, one could only act on $\cH$ (respectively $\cK$) on the right (respectively left), which causes problems with the definition of the Connes fusion tensor product, since it is easy to see that $\xi q\ot\eta = \xi\ot q\eta$ for any $q\in Q$ and bounded vectors $\xi\in\cH$, $\eta\in \cK$, but $Q$ is not necessarily abelian. Moreover if $\cH$ and $\cK$ were $Q$-bimodules, $\cH\ot_Q \cK$ would not be a (right or left) $Q$-module in such a way as to identify the original (right, respectively left) actions on $\cH$ and $\cK$ used to construct $\cH\ot_Q \cK$. 

Nevertheless, the above approach works when $\cH=L^2(N_1)$ and $\cK=L^2(N_2)$ for tracial von Neumann algebras $N_1$, $N_2$ with $Q$ as a common subalgebra \textit{of the center}. In this case we have well-defined actions of $N_1$ and $N_2^{\op}$ on $L^2(N_1)\ot_Q L^2(N_2)$ by
\begin{equation}\label{eq:relprod}
	x\cdot (\xi\ot_Q \eta) \coloneqq x\xi \ot_Q \eta, \qquad\text{and}\qquad (\xi\ot_Q \eta)\cdot y \coloneqq \xi\ot_Q \eta y
\end{equation}
for $x\in N_1$, $y\in N_2$, $\xi\in L^2(N_1)$, $\eta\in L^2(N_2)$, and the two actions of $Q$ coincide by the construction of the Connes fusion tensor product.

\begin{remark}
For general $Q\subset N$, we can look at the ``doubling'' $L^2(N)\ot_Q L^2(N)$. In this case, it does come from a von Neumann algebra that can be viewed as a canonical ``relatively independent product'' of $N$ with itself over the subalgebra $Q$. This von Neumann algebra is given by the aforementioned basic construction $\langle N,e_Q\rangle\subset B(L^2(N))$ generated by $N$ and the projection $e_Q$ onto the subspace $L^2(Q)$. However, as observed in the explanation preceding Lemma~\ref{lem:HS}, $\langle N,e_Q\rangle$ is in general not a tracial von Neumann algebra.
\end{remark}

\begin{definition}\label{def-relprod}
Suppose $Q$ is a common subalgebra of the centers of $N_1$ and $N_2$. We will call the von Neumann subalgebra of $B(L^2(N_1)\ot_Q L^2(N_2))$ generated by $N_1$ and $N_2$ for the actions defined by \eqref{eq:relprod} the \textit{relatively independent product of $N_1$ and $N_2$ over $Q$} and denote it by $N_1\otb_Q N_2$.
\end{definition}

\begin{remark}
If $Q$ is a common subalgebra of the centers of von Neumann algebras $N_1$, $N_2$, and $N_3$, we observe that we have canonical isomorphisms 
\[
(N_1\otb_Q N_2)\otb_Q N_3 \cong N_1\otb_Q (N_2\otb_Q N_3) \eqqcolon N_1\otb_Q N_2\otb_Q N_3,
\]
and
\[
N_1\otb_Q N_2\cong N_2\otb_Q N_1.
\]
\end{remark}

Assuming the above situation, we will next prove the extension of a classical result from ergodic theory, characterizing the maximal compact extension inside a relatively independent product, which we formulated as Theorem \ref{thm:RelProd} in the introduction. In the first half of the proof, we follow very closely the original proof of Furstenberg \cite[Theorem~7.1]{furstenberg-szemeredi} (see also \cite[Theorem~9.21]{glasner}). However, we provide the proof on a purely algebraic level, whereas the original proof is ``point-based'', and we avoid the use of disintegration.

\begin{proof}[Proof of Theorem \ref{thm:RelProd}]
Note that the maximal compact extensions in the statement of the theorem indeed exist thanks to Proposition \ref{prop:quasireg}, by observing that any subalgebra of the center is quasi-regular. It is easy to see that $P_1\otb_Q P_2\subset P$, and we proceed with showing the reverse inclusion. 

Assume $f\in P\ominus P_1\otb_Q P_2$ is an element such that the $\Gamma$-invariant $Q$-module $V$ generated by $f$ is of finite rank, and let $f_1, \ldots, f_r$ be a $Q$-orthonormal basis for $V$. Then for every $1\leq i\leq r$, $f_i\in P\ominus P_1\otb_Q P_2$, and if we write $\gamma \cdot f_i = \sum_{j=1}^r \lambda_{ij}(\gamma) f_j$, then $(\lambda_{ij}(\gamma))_{i,j=1}^r$ is a unitary matrix over $Q$ for every $\gamma\in\Gamma$. We now define the element
\[
\psi\coloneqq (\id\ot\id\ot E_Q)\left(\sum_{i=1}^r \iota_{13}(f_i)\iota_{23}(f_i^*)\right) \in L^2(N_1)\ot_Q L^2(N_1)\ot_Q L^2(Q) \cong L^2(N_1)\ot_Q L^2(N_1),
\]
where $\iota_{13},\iota_{23}: N_1\otb_Q N_2\rarrow N_1\otb_Q N_1\otb_Q N_2$ denote the canonical embeddings determined by $\iota_{13}(x\ot_Q y) = x\ot 1\ot y$ and $\iota_{23}(x\ot_Q y) = 1\ot x\ot y$. For every $\gamma\in\Gamma$, using that $(\lambda_{ij}(\gamma))_{i,j=1}^r$ is a unitary matrix over $Q$, a straightforward calculation then gives that $\gamma\cdot \psi = \psi$. Therefore, we can view $\psi\in L^2(N_1)\ot_Q L^2(N_1)$ as a $\Gamma$-invariant conditional Hilbert-Schmidt operator. As in the proof of Theorem~\ref{thm:compchar}, we get that the image of $\psi$ is spanned by the $\psi$-invariant and $\Gamma$-invariant $Q$-finitely generated submodules. Take such a submodule $W$ and let $\phi_1, \ldots, \phi_s$ be a $Q$-orthonormal basis for $W$. For $1\leq i\leq r$, $1\leq j\leq s$, we define the elements
\[
g_{ij}\coloneqq (E_Q\ot\id)(f_i (\phi_j^*\ot 1)) \in L^2(Q)\ot_Q L^2(N_2) \cong L^2(N_2).
\]
Writing $\gamma\cdot \phi_i = \sum_{j=1}^s \eta_{ij}(\gamma)\phi_j$, a straightforward computation yields
\[
\gamma\cdot g_{ij} = \sum_{m=1}^r \sum_{n=1}^s \lambda_{im}(\gamma)\eta_{jn}(\gamma)^* g_{mn},
\]
for every $\gamma\in\Gamma$. Hence $g_{ij}\in L^2(P_2)$. By construction, we then have for every $1\leq i\leq r$, $1\leq j,k\leq s$ that
\[
\phi_k\ot_Q g_{ij} \in L^2(P_1)\ot_Q L^2(P_2),
\]
and thus $f_i$ is perpendicular to $\phi_k\ot_Q g_{ij}$. Hence for $1\leq j,k\leq s$, we have
\begin{align*}
0&= (E_Q\ot E_Q)\left(\sum_{i=1}^r f_i (\phi_k^*\ot_Q g_{ij}^*)\right)\\
&= (E_Q\ot E_Q)\left(\sum_{i=1}^r f_i (\phi_k^*\ot_Q 1)(1\ot_Q g_{ij}^*)\right)\\
&= (E_Q\ot E_Q)\left(\sum_{i=1}^r f_i [(\phi_k^*\ot_Q 1)][1\ot (E_Q\ot\id)(f_i^* (\phi_j\ot 1))])\right).
\end{align*}
where in the last equality we used the definition of $g_{ij}$ considered inside $L^2(N_2)$. Interpreting the elements between the large brackets as living in $L^2(N_1)\ot_Q L^2(Q)\ot_Q L^2(N_2)$ ($\cong L^2(N_1) \ot_Q L^2(N_2)$), we can write
\begin{align*}
(E_Q\ot E_Q)&\left(\sum_{i=1}^r f_i [(\phi_k^*\ot_Q 1)][1\ot (E_Q\ot\id)(f_i^* (\phi_j\ot 1))]\right)\\
&= (E_Q\ot \id\ot E_Q)\left(\sum_{i=1}^r \iota_{13}(f_i (\phi_k^*\ot_Q 1))\iota_{23}((E_Q\ot\id)(f_i^* (\phi_j\ot 1)))\right)\\
&= (E_Q\ot E_Q\ot E_Q)\left(\sum_{i=1}^r \iota_{13}(f_i (\phi_k^*\ot_Q 1))\iota_{23}(f_i^* (\phi_j\ot 1)))\right).
\end{align*}
Taking together terms appropriately and observing that $\iota_{13}(\phi_k^*\ot_Q 1) = (\phi_k^*\ot_Q 1\ot_Q 1)$ commutes with $\iota_{23}(f_i^*)$ yields
\begin{align*}
(E_Q\ot E_Q&\ot E_Q)\left(\sum_{i=1}^r \iota_{13}(f_i (\phi_k^*\ot_Q 1))\iota_{23}(f_i^* (\phi_j\ot 1)))\right)\\
&= (E_Q\ot E_Q\ot E_Q) \left(\left[\sum_{i=1}^r (\id\ot\id\ot E_Q)(\iota_{13}(f_i)\iota_{23}(f_i^*))\right] (\phi_k^*\ot_Q\phi_j\ot_Q 1)\right).
\end{align*}
Re-interpreting the elements between brackets to be inside $L^2(N_1)\ot_Q L^2(N_1)$ ($\cong L^2(N_1)\ot_Q L^2(N_1)\ot_Q L^2(Q)$), and using the definition of $\psi$, we get
\begin{align*}
(E_Q\ot E_Q\ot E_Q) &\left(\left[\sum_{i=1}^r (\id\ot\id\ot E_Q)(\iota_{13}(f_i)\iota_{23}(f_i^*))\right] (\phi_k^*\ot_Q\phi_j\ot_Q 1)\right)\\
&= (E_Q\ot E_Q) (\psi(\phi_k^*\ot_Q\phi_j)).
\end{align*}
Finally, applying the definitions of the conditional convolution \eqref{def:conditionalconvolution} and the $Q$-inner product, we get
\begin{align*}
(E_Q\ot E_Q) (\psi(\phi_k^*\ot_Q\phi_j))&= E_Q((\id\ot E_Q)(\psi(1\ot \phi_j)) \phi_k^*)\\
&= \langle \psi\ast_Q \phi_j,\phi_k\rangle_Q.
\end{align*}
Following the above string of equalities, we conclude that for any $1\leq j,k\leq s$ we have $\langle \psi\ast_Q \phi_j,\phi_k\rangle_Q = 0$. Hence $\psi|_W=0$, and since $W$ was arbitrary, it follows that $\psi=0$.

We next prove that this implies that $f_i=0$ for every $1\leq i\leq r$, and hence $f=0$, which will finish the proof. Fix a $Q$-orthonormal basis $\{e_j\}_{j\in J}$ of $L^2(N_2)$ as a right $Q$-module, and for $1\leq i\leq r$ write
\[
f_i = \sum_j f_i^{(j)}\ot_Q e_j.
\]
Then we get 
\[
\psi = \sum_{i=1}^r \sum_{j,k} f_i^{(j)}\ot_Q f_i^{(k)*} E_Q(e_je_k^*) = \sum_{i=1}^r \sum_{j} f_i^{(j)}\ot_Q f_i^{(j)*}.
\]
As a conditional Hilbert-Schmidt operator, this essentially means that $\psi$ is written as a sum of rank 1 operators over $Q$ given by $M_{f_i^{(j)}}M_{f_i^{(j)}}^*$, where $M_{f_i^{(j)}}: L^2(Q)\rarrow L^2(N_1)$ denotes the $Q$-linear operator given by multiplication by $f_i^{(j)}$. Following \cite[Proposition~8.4.2]{AP}, we then get
\begin{align*}
0 &= \hat\tau(\psi)\\
&= \hat\tau(\sum_{i=1}^r \sum_{j} f_i^{(j)}\ot_Q f_i^{(j)*})\\
&= \hat\tau(\sum_{i=1}^r \sum_{j} M_{f_i^{(j)}}M_{f_i^{(j)}}^*)\\
&= \tau(\sum_{i=1}^r \sum_{j} M_{f_i^{(j)}}^*M_{f_i^{(j)}})\\
&= \sum_{i=1}^r \sum_{j} \tau_Q(E_Q(f_i^{(j)^*}f_i^{(j)}))\\
&= \sum_{i=1}^r \sum_{j} \tau_{N_1}(f_i^{(j)^*}f_i^{(j)}).
\end{align*}
We conclude that for all $i,j$, $f_i^{(j)}=0$, and thus every $f_i=0$, which finished the proof of Theorem~\ref{thm:RelProd}.
\end{proof}

As a special case of Theorem \ref{thm:RelProd}, we deduce the well-known result of Furstenberg (\!\!\cite[Theorem~7.1]{furstenberg-szemeredi}, see also \cite[Theorem~9.21]{glasner}) by which the proof is inspired, but where we can remove any countability/separability assumptions.

\begin{theorem}\label{thm:RelProdComm}
	Let $(X_1,\mu_1,\sigma_1)$ and $(X_2,\mu_2, \sigma_2)$ be probability algebra dynamical $\Gamma$-systems with factor maps $(X_1,\mu_1,\sigma_1)\to (Y,\nu)$ and $(X_2,\mu_2,\sigma_2)\to (Y,\nu)$. Denote by $\rho_i:Z_i\rarrow Y$ the maximal compact extension below $X_i$. Consider the canonical factor map $\pi: X_1\times_Y X_2\rarrow Y$, and denote by $\rho:Z\rarrow Y$ the maximal compact extension below $X_1\times_Y X_2$. Then
	\[
	Z\cong Z_1\times_Y Z_2.
	\]
\end{theorem}

\section{Weakly mixing extensions and dichotomies}\label{sec-dichotomy}

Weakly mixing extensions are the complementary notion to compact extensions. These have been well-studied in the von Neumann algebraic setting and have for instance been used extensively in Popa's deformation/rigidity theory, e.g., see \cite{popa-betti,popa-2,popa-1,popa,popa-vaes}. The following proposition due to Popa (see \cite[Lemma 2.1]{popa}) lists some equivalent characterizations of weakly mixing extensions for tracial $W^*$-dynamical systems. Note that item (3) in Proposition \ref{prop:PopaWM} complements property (2) in Theorem \ref{thm:compchar}, which we will use below to deduce some dichotomies analogous to the commutative case. 

\begin{proposition}\label{prop:PopaWM}
Let $\Gamma\act (N,\tau)$ be a tracial $W^*$-dynamical system, and let $Q\subset N$ be a $\Gamma$-invariant von Neumann subalgebra. Then the following properties are equivalent. 
    \begin{itemize}
		\item[(1)] For any finite set $F\subset N\ominus Q$ and every $\eps>0$ there exists $\gamma\in \Gamma$ such that $$\|E_Q(f\sigma_\gamma(g))\|_2\leq \eps$$ for all $f,g\in F$. 
		\item[(1')] For any finite set $F\subset L^2(N)\ominus L^2(Q)$ with $E_Q(ff^*)$ bounded for all $f\in F$ and every $\eps>0$ there exists $\gamma\in \Gamma$ such that $$\|E_Q(f\sigma_\gamma(g))\|_2\leq \eps$$ for all $f,g\in F$. 
		\item[(2)] There exists a net $(\gamma_i)_{i\in I}$ in $\Gamma$, with the property that for any finite set $\Lambda\subset \Gamma$ there exists $i_\Lambda$ such that $\gamma_i\not\in \Lambda$ for all $i\geq i_\Lambda$, and an orthonormal basis $\{1\}\cup \{f_j\}_{j\in J}$ of $L^2(N)$ relative to $Q$ such that $$\lim_{i} \norm{E_Q(f^*_j \sigma_{\gamma_i}(f_{j'}))}_2\to 0$$
		for all $j,j'\in J$.  
		\item[(3)] Any $\Gamma$-invariant $K\in L^2(N)\ot_Q L^2(N)$ lies in $L^2(Q)$. 
	\end{itemize}
\end{proposition}

\begin{proof}
We note that \cite[Lemma 2.10]{popa} is stated and proved for countable groups acting on separable tracial von Neumann algebras. The same proof works for arbitrary discrete groups acting on arbitrary tracial von Neumann algebras by replacing sequences with nets in (2) and observing that any right $Q$-module still has a (not necessarily countable) orthonormal basis (cf. \cite{AP}).
\end{proof}

\begin{definition}
	If any (and thus all) of the properties in Proposition \ref{prop:PopaWM} is satisfied, then we say that $N$ is \emph{weakly mixing relative to} $Q$, or that $N$ is a \emph{weakly mixing extension} of $Q$. 
	
Given any tracial $W^*$-dynamical inclusion $\Gamma\act (Q\subset N)$, we denote
\[
\cWM^2_{Q,N}\coloneqq L^2(N)\ominus \cAP^2_{Q,N},
\]
and
\[
\cWM_{Q,N}\coloneqq \cWM^2_{Q,N}\cap N.
\]
Obviously, $\cWM^2_{Q,N}$ is a $\Gamma$-invariant Hilbert subspace of $L^2(N)$.
\end{definition}

By Proposition \ref{prop:PopaWM} and Theorem \ref{thm:compchar}, we can deduce some compact/weak mixing dichotomies. The Hilbert space $L^2(N)$ splits as an orthogonal sum of its relatively compact and weakly mixing parts without any further assumption on the inclusion, similarly to the commutative situation. However on the von Neumann algebra level we need to assume that the inclusion $Q\subset N$ is quasi-regular.

\begin{proposition}[$L^2$-dichotomy]
Let $\Gamma\act (N,\tau)$ be a tracial $W^*$-dynamical system, and let $Q\subset N$ be a $\Gamma$-invariant von Neumann subalgebra. Then $\cWM^2_{Q,N}$ is weakly mixing relative to $Q$, in the sense that property (1') from Proposition \ref{prop:PopaWM} holds with $L^2(N)\ominus L^2(Q)$ replaced by $\cWM^2_{Q,N}$. In particular, we can decompose $L^2(N)$ as
	\[
	L^2(N) = \cAP^2_{Q,N} \oplus \cWM^2_{Q,N}
	\]
	into a relatively compact and a relatively weakly mixing part.
\end{proposition}
\begin{proof}
One can mimic the proof of $(3)\Rightarrow (1')$ from \cite[Lemma 2.10]{popa} with $L^2(N)\ominus L^2(Q)$ replaced by $\cWM^2_{Q,N}$, and reach a contradiction with property (2) in Theorem \ref{thm:compchar} if $\cWM^2_{Q,N}$ wouldn't satisfy property (1'). We leave the details to the interested reader.
\end{proof}

\begin{corollary}
Let $\Gamma\act (N,\tau)$ be a tracial $W^*$-dynamical system, and let $Q\subset N$ be a $\Gamma$-invariant von Neumann subalgebra. If the inclusion $Q\subset N$ is quasi-regular, then the space of relatively weakly mixing elements $\cWM_Q\subset N$ is given by $\cWM_Q = (\cAP_Q)^\perp = \{x\in N\mid E_{\cAP_Q}(x)=0\}$.
\end{corollary}

Next we state a version of the notion of \textit{disjointedness} in the non-commutative setting. By analogy with the commutative setting, we will then use it together with our Theorem \ref{thm:compchar} to characterize weakly mixing extensions.

\begin{definition}
	Suppose $\Gamma\act (M,\tau_M)$ and $\Gamma\act (N,\tau_N)$ are tracial $W^*$-dynamical systems and $Q\subset M,N$ is a common $\Gamma$-invariant subalgebra. We say that $M$ and $N$ are \textit{disjoint over $Q$} if $\tau_M\ot_Q\tau_N$ is the only joining of $M$ and $N$ over $Q$.
\end{definition}

In Proposition~\ref{thm-disjoint}, we will characterize weakly mixing extensions as being exactly those that are disjoint from all compact extensions (under the assumption of quasi-regularity). The special case $Q=\C 1$ for general (not necessarily tracial) $W^*$-dynamical systems was established for abelian groups by Duvenhage in \cite{duvenhage-joining2} and for general locally compact groups by Bannon, Cameron and Mukherjee in \cite[Theorem 6.15]{joining-1}. Our approach offers an efficient proof for this fact in the tracial setting. For example, in \cite{joining-1} it had to be checked by hand that relatively almost periodic elements form a von Neumann subalgebra, which is immediate from the characterization in Lemma \ref{lem:cap} when $Q=\C 1$.

\begin{proposition}\label{thm-disjoint}
Let $\Gamma\act (N,\tau)$ be a tracial $W^*$-dynamical system, and let $Q\subset N$ be a quasi-regular and $\Gamma$-invariant von Neumann subalgebra.  Then $N$ is a weakly mixing extension of $Q$ if and only if it is disjoint from every system which is a compact extension of $Q$.
\end{proposition}

We split the proof into two lemmas. The first one is essentially an adaptation of the argument in \cite[Theorem 6.8]{joining-1}. We provide the proof for the sake of completeness, as it considerably simplifies thanks to being in the tracial setting and the characterization of almost periodic elements in Lemma \ref{lem:cap}. Note that the condition that the inclusion $Q\subset N$ is quasi-regular is not needed for this direction.

\begin{lemma}\label{lem:disjoint}
	Suppose $\Gamma\act (Q\subset N)$ is a relatively weakly mixing tracial $W^*$-dynamical inclusion. Then it is disjoint from every system which is a compact extension of $Q$.
\end{lemma}
\begin{proof}
Suppose $\Gamma\act (Q\subset M)$ is a relatively compact tracial $W^*$-dynamical system which is not disjoint from $\Gamma\act (Q\subset N)$, and let $\varphi:M\algot N^{\op}\rightarrow\C$ be a non-trivial joining. Consider the $M$-$N$-bimodule $L^2(M\algot N^{\op},\varphi)$ which is the separation-completion of $M\algot N^{\op}$ for $\varphi$, and denote by $\xi_\varphi$ the class of $1_M\ot 1_N$. Let $\Phi:M\rarrow N$ denote the corresponding completely positive map (see Section~\ref{sec:bimod}), and denote by $T_\Phi:L^2(M)\rarrow L^2(N)$ the associated linear operator. Observe that by construction, we have that $\Phi|_Q = \id_Q$. Denoting by $\pi_M:\Gamma\rarrow \cU(L^2(M))$ and $\pi_N:\Gamma\rarrow \cU(L^2(N))$ the corresponding Koopman representations of the given actions, we see that moreover
	\[
	T_\Phi \circ \pi_M = \pi_N \circ T_\Phi.
	\]
	It is not difficult to check that $T_\Phi(L^2(M))\subseteq L^2(Q)$ if and only if $\varphi$ was the relatively independent joining $\tau_M\ot_Q \tau_N$ (since in that case, necessarily $T_\Phi = E_Q$ as it is a completely positive projection onto $Q$). Since we are assuming this is not the case, we can find $\xi\notin L^2(Q)$ such that $T_\Phi(\xi)\notin L^2(Q)$. Since $M$ is a compact extension of $Q$, this means that $\xi$ is relatively almost periodic, that is,  for every $\eps>0$, there exist $\eta_1, \ldots, \eta_k\in L^2(M)$ such that for every $\gamma\in \Gamma$, there exist $\kappa_i(\gamma)\in Q$ with $\sup_{i,\gamma}\norm{\kappa_i(\gamma)}<\infty$ and 
	\[
	\norm{\gamma\cdot \xi - \sum_{i=1}^k \eta_i\kappa_i(\gamma)}_2 \leq \eps.
	\]
	Using that $T_\Phi|_{L^2(Q)} = \id$, it is then a straightforward calculation to check that for any given $\gamma\in\Gamma$ we have
	\[
	\norm{\gamma\cdot T_\Phi(\xi) - \sum_{i=1}^k T_\Phi(\eta_i)\kappa_i(\gamma)}_2 \leq \eps,
	\]
	in other words $T_\Phi(\xi)\in L^2(N)$ is relatively almost periodic, contradicting the fact that $N$ is a weakly mixing extension of $Q$.
\end{proof}

\begin{lemma}
Suppose the tracial $W^*$-dynamical inclusion $\Gamma\act (Q\subset N)$, with $Q\subset N$ quasi-regular, is disjoint from every system which is a compact extension of $Q$. Then $N$ is a weakly mixing extension of $Q$.
\end{lemma}
\begin{proof}
Assume $N$ is not a weakly mixing extension of $Q$, and use Proposition~\ref{prop:quasireg} to find a subsystem $P\subset N$ such that $P$ is a compact extension of $Q$. Then it is easily seen that $\varphi: P\algot N^{\op}\rarrow \C$ defined by 
	\[
	\varphi(x\ot y) = \Delta_P(x\ot E_P(y))
	\]
is a joining of $P$ and $N$ over $Q$ which is different from $\tau_P\ot_Q\tau_N$.
\end{proof}

\section{Cubic systems and the Host--Kra--Ziegler tower of compact extensions}\label{sec-hk}

\textbf{Convention.} Throughout this section, $\Gamma$ will denote a countable discrete abelian group, and all von Neumann algebras are assumed to have separable predual.

Following closely the constructions in \cite{host-kra}, we will introduce Host--Kra cubic systems associated to an ergodic $W^*$-dynamical system, which will provide us with a tower of compact extensions, generally different from, but somewhat finer than, the Furstenberg--Zimmer tower of maximal compact extensions. Moreover, contrary to the case of maximal compact extensions, we don't need any additional assumptions on the inclusions $Q\subset N$ to make sure the subsequent extensions are von Neumann subalgebras, thus giving a complete analogue of the Host--Kra--Ziegler tower of compact extensions from the commutative case.

We will be dealing with $2^k$-tensor powers of von Neumann algebras, and we start with introducing some notation. Let $k\geq 0$ be an integer. Denote $V_k\coloneqq \{0,1\}^k$ and $N^{[k]}\coloneqq N^{\ot 2^k}$ for a von Neumann algebra $N$. We will use elements $\epsilon$ of $V_k$ to index elements $x=(x_\epsilon)_{\epsilon\in V_k}$ of $N^{[k]}$. Given a map $f:M\rarrow N$ between von Neumann algebras, we denote by $f^{[k]}:M^{[k]}\rarrow N^{[k]}$ the componentwise application of $f$, that is, $(f^{[k]}(x))_\epsilon = f(x_\epsilon)$.

Note that we can identify $N^{[k+1]}$ with $N^{[k]}\ot N^{[k]}$. By convention, we will do this based on the last coordinate, writing $x = (x',x'')$ for $x\in N^{[k+1]}$ where 
\[
x'_\epsilon = x_{\epsilon 0} \quad \text{and} \quad x''_\epsilon = x_{\epsilon 1}.
\]

For the remainder of the section, we fix an ergodic action $\Gamma\act^\sigma (N,\tau)$. For each $k\geq 0$ we can consider the diagonal action $\sigma^{[k]}$ of $\Gamma$ on $N^{[k]}$, and we denote by $\cI^{[k]}\subset N^{[k]}$ the subalgebra of $\Gamma$-invariant vectors. We will next define by induction a $\Gamma$-invariant state $\tau^{[k]}$ on $N^{[k]}$, which will essentially arise as an iterated joining over the $\cI^{[j]}$. Before we do this, we make the following observation.

\begin{lemma}\label{lem:hyperfinite}
    For every $k\geq 0$, $\cI^{[k]}$ is a hyperfinite von Neumann algebra.
\end{lemma}
\begin{proof}
By the well-known theorem of H\o egh-Krohn, Landstad and St\o rmer \cite{HLS}, we know that if $\Gamma$ admits an ergodic compact action on a von Neumann algebra $M$, then $M$ is necessarily hyperfinite. Therefore the maximal compact part $\cAP_N\subset N$ is hyperfinite.

Since $\Gamma\act (N,\tau)$ is ergodic, we firstly have that $\cI^{[0]}$ is trivial, and therefore hyperfinite. Let $k\geq 1$. Then by Proposition~\ref{cor:invariantvectors}, 
\[
\cI^{[k]} \subset \cAP_{N^{[k-1]}}\ot \cAP_{N^{[k-1]}}.
\]
By Theorem~\ref{thm:RelProd} and induction, we moreover have
\[
\cAP_{N^{[k-1]}} = (\cAP_N)^{\ot 2^{k-1}},
\]
and hence
\[
\cI^{[k]} \subset (\cAP_N)^{\ot 2^{k}}.
\]
Since $\cAP_N$ is hyperfinite, the same holds for $\cI^{[k]}$.
\end{proof}

\begin{definition}
Let $\tau^{[0]}=\tau$. Assuming $\tau^{[k]}$ is defined, we define $\tau^{[k+1]}$ on $N^{[k+1]}\cong N^{[k]}\otb N^{[k]}$ by
\[
\tau^{[k+1]}(x\ot y) \coloneqq \tau^{[k]}(E_{\cI^{[k]}}(x)E_{\cI^{[k]}}(y)).
\]
Note that $\tau^{[k+1]}:N^{[k+1]}\rarrow\C$ is well-defined, since by Lemma~\ref{lem:hyperfinite}, $\cI^{[k]}$ is hyperfinite, and hence by Connes' theorem \cite[Theorem~5.1]{connes-hyperfinite}, the map $\cI^{[k]}\algot (\cI^{[k]})^\op\rarrow \cI^{[k]}: a\ot b\mapsto ab$ is min-continuous.
\end{definition}

\begin{remark}
    Note that the definition essentially states that $\tau^{[k+1]}$ is the relatively independent joining of $(N^{[k]},\tau^{[k]})$ with itself over $\cI^{[k]}$. However, we do not use this terminology here, because we have only defined joinings for \textit{tracial} von Neumann algebras in this paper.
\end{remark}

We note that since the system $\Gamma\act^\sigma (N,\tau)$ is assumed to be ergodic, $\cI^{[0]}$ is trivial and $\tau^{[1]} = \tau\ot\tau$. Also, if the system is weakly mixing, then each $\cI^{[k]}$ is trivial and $\tau^{[k]} = \tau^{\ot 2^k}$ for every $k\geq 0$. On the other hand, if the system is not weakly mixing, then the $\cI^{[k]}$ are not trivial for $k\geq 1$, and reveal some information about the ``compact parts'' of the system.

Before continuing, we introduce the so-called \textit{side transformations}. For this, note that we can view $V_k$ as indexing the vertices of the $k$-cube, and so we can talk about its faces. More precisely, for $0\leq l\leq k$, $J\subset \{1, \ldots, k\}$ with $\abs{J}=k-l$, and $\eta\in \{0,1\}^J$, we call the subset
\[
\alpha\coloneqq \{\epsilon\in V_k\mid \epsilon_j=\eta_j \text{ for every } j\in J\}
\]
a face of dimension $l$ of $V_k$, or an $l$-face. We also call the faces of dimension $k-1$ the \textit{sides} of $V_k$.

Let $\alpha$ be a face of $V_k$ and for every $\gamma\in\Gamma$, denote by $\sigma_{\gamma,\alpha}^{[k]}:N^{[k]}\rarrow N^{[k]}$ the map given by
	\begin{align*}
		(\sigma_{\gamma,\alpha}^{[k]} x)_\epsilon = 
		\begin{cases}
			\sigma_\gamma(x_\epsilon) &\text{if } \epsilon\in\alpha,\\
			x_\epsilon &\text{otherwise}.
		\end{cases}
	\end{align*}
	We call $\sigma_{\gamma,\alpha}^{[k]}$ a face transformation. If $\alpha$ is a side we also call it a side transformation. We denote by $\cT_{k-1}^{[k]}$ the group generated by all side transformations, that is, 
	\[
	\cT_{k-1}^{[k]} \coloneqq \langle\sigma_{\gamma,\alpha}\mid \gamma\in \Gamma, \alpha \text{ a side of } V_k\rangle
	\]
	The subgroup generated by those side transformations $\sigma_{\gamma,\alpha}$ where $\alpha$ is a side not containing $0$ will be denoted by $\cT_*^{[k]}$.

\begin{lemma}\label{lem:invTk-1}
For every $k\geq 1$, the state $\tau^{[k]}$ is invariant for the action of $\cT_{k-1}^{[k]}$ on $N^{[k]}$. 
\end{lemma}
\begin{proof}
The proof from \cite[Lemma 3.3]{host-kra} applies mutatis mutandis. Note that we need the commutativity of $\Gamma$ here. 
\end{proof}

Let $\cJ^{[k]}\subset N^{[k]}$ denote the subalgebra of invariant vectors under the action of $\cT_*^{[k]}$, and denote by $L^2(N^{[k]},\tau^{[k]})$ the GNS-construction of $N^{[k]}$ with respect to the state $\tau^{[k]}$.

\begin{proposition}\label{prop:zerocoord}
Inside $L^2(N^{[k]},\tau^{[k]})$, we have that $L^2(\cJ^{[k]},\tau^{[k]}) = L^2(N,\tau)\ot 1^{\ot 2^k-1}$, where the first factor corresponds to the coordinate $0\in V_k$.
\end{proposition}
\begin{proof}
Clearly, any element from $L^2(N)\ot 1^{\ot 2^k-1}$ is invariant under the action of $\cT_*^{[k]}$. We proceed to prove the reverse inclusion by induction. For $k=1$, we have $N^{[1]}=N\otb N$ and $\cT_*^{[k]}$ contains $\id\times \sigma_\gamma$ for $\gamma\in\Gamma$, which easily implies the result. Now let $k\geq 1$ and suppose $x\in \cJ^{[k+1]}$. Write $x=\sum_i x_{1,i}\ot x_{2,i}$ where for each $i$, $x_{1,i},x_{2,i}\in N^{[k]}$. Let $\alpha$ be the face $\{0,1\}^k\times \{1\}$ of $V_{k+1}$. Hence for any $\gamma\in\Gamma$, we have $\sigma_{\gamma,\alpha}^{[k+1]} = \id_{N^{[k]}}\times \sigma_{\gamma}^{[k]}$. Denote by $I_\alpha\supset \cJ^{[k+1]}$ the subalgebra of invariant vectors under the action of $\sigma_{\gamma,\alpha}^{[k+1]}$, $\gamma\in\Gamma$. From the above, we observe that we have $I_\alpha = N^{[k]}\ot \cI^{[k]}$.
We can then compute
\begin{align*}
	E_{I_\alpha}(x) &= \sum_i E_{I_\alpha}(x_{1,i}\ot x_{2,i}) \\
	&= \sum_i x_{1,i}\ot E_{\cI^{[k]}}(x_{2,i})
\end{align*}
By the definition of $\tau^{[k+1]}$, the latter is equal to $\sum_i x_{1,i} E_{\cI^{[k]}}(x_{2,i})\ot 1$ in the GNS-construction of $N^{[k+1]}$ with respect to $\tau^{[k+1]}$. Since $x\in \cJ^{[k+1]}\subset I_\alpha$ and hence $x=E_{I_\alpha}(x)$, we deduce that 
\[
x = x'\ot 1
\]
for some $x'\in L^2(N^{[k]})$. It is then easy to see that $x'\in L^2(\cJ^{[k]})$, and by induction it thus follows that $x'\in  L^2(N)\ot 1^{\ot 2^k-1}$, which finishes the proof.
\end{proof}

\subsection{Construction of the compact extensions arising from cubic systems}

First, we establish some additional notation, mostly following \cite[Section 4]{host-kra}. We write $V_k^*\coloneqq V_k\setminus \{0\}$, $N^{[k]^*}\coloneqq \ot_{\epsilon\in V_k^*} N$, and $\tau^{[k]^*}$ the restriction of $\tau^{[k]}$ to $N^{[k]^*}$. We represent a point $x\in N^{[k]}$ as $x=(x_0,\tilde x)$, where $x_0\in N$ and $\tilde x\in N^{[k]^*}$, and also write $N^{[k]} = N_0\ot N^{[k]^*}$, where $N_0\cong N$ is the copy of $N$ at coordinate $0\in V^{[k]}$. Next, we note that we can restrict the actions of $\cT_{k-1}^{[k]}$ and its subgroup $\cT_*^{[k]}$ to $N_0$ and $N^{[k]^*}$.
It thus makes sense to define the sets
\[
\cI^{[k]^*}\coloneqq \Inv(\cT_{k-1}^{[k]}|_{N^{[k]^*}})\subset N^{[k]^*},
\]
and
\[
\cJ^{[k]^*}\coloneqq \Inv(\cT_{*}^{[k]}|_{N^{[k]^*}})\subset N^{[k]^*}.
\]
By construction, if $x\in\cJ^{[k]^*}$ then $1\ot x\in\cJ^{[k]}$. Using Proposition \ref{prop:zerocoord}, we can thus find $x_0\in N$ such that $\norm{1\ot x - x_0\ot 1^{\ot 2^k-1}}_{\tau^{[k]}} = 0$. This observation leads us to the following definition.

\begin{definition}
	Given an integer $k\geq 1$, we define $\cZ_{k-1}\subset N$ as
	\[
	\cZ_{k-1} \coloneqq \{x_0\in N\mid \exists x\in N^{[k]^*}\text{ such that } \norm{1\ot x - x_0\ot 1^{\ot 2^k-1}}_{\tau^{[k]}} = 0\}.
	\]
\end{definition}

We note that $\cZ_{k-1}\subset N$ is a von Neumann subalgebra which can be identified with $\cJ^{[k]^*}\subset N^{[k]^*}$. Also, it is easy to see that $\cZ_{k-1}$ is $\Gamma$-invariant, and therefore provides us with a subsystem of the original action $\Gamma\act N$. A first easy observation is that these subsystems form an increasing sequence.

\begin{lemma}\label{lem:cZincreasing}
	For every $k\geq 1$, $\cZ_{k-1}\subseteq\cZ_{k}$, that is,  the $\cZ_k$ form an increasing sequence of subsystems.
\end{lemma}
\begin{proof}
If $x_0\in\cZ_{k-1}$, then there exists $x\in N^{[k]^*}$ such that
\[
\norm{1\ot x - x_0 \ot 1^{\ot 2^k-1}}_{\tau^{[k]}} = 0.
\]
Using the definition of $\tau^{[k+1]}$, it is then a straightforward calculation to check that with $\tilde x\coloneqq x\ot 1^{\ot 2^k}\in N^{[k+1]^*}$,
\[
\norm{1\ot \tilde x - x_0 \ot 1^{\ot 2^{k+1}-1}}_{\tau^{[k+1]}} = 0.
\]
Hence $x_0\in\cZ_k$.
\end{proof}

We proceed with the following observation.

\begin{lemma}[\!\!{\cite[Lemma 4.2]{host-kra}}]\label{lem:NotNstar}
Fix $k\geq 1$. Then for every $x\in N$ and $y\in N^{[k]^*}$, we have
\[
\tau^{[k]}(x\ot y) = \tau(E_{\cZ_{k-1}}(x)E_{\cJ^{[k]^*}}(y)),
\]
where we view $E_{\cJ^{[k]^*}}(y)\in \cJ^{[k]^*} \cong\cZ_{k-1}\subset N$.
\end{lemma}  
\begin{remark}
    Intuitively, we can view this result as stating that $(N^{[k]},\tau^{[k]})$ is the relatively independent joining of $(N,\tau)$ and $(N^{[k]^*},\tau^{[k]^*})$ over $\cZ_{k-1}$ (which is identified with $\cJ^{[k]^*}\subset N^{[k]^*}$).
\end{remark}
\begin{proof}
The proof is essentially an adaptation of the proof of \cite[Lemma 4.2]{host-kra}. Let $x\in N$, and $y\in N^{[k]^*}$. By Lemma \ref{lem:invTk-1}, 
\[
\tau^{[k]}(x\ot y) = \tau^{[k]}(x\ot \sigma(y))
\]
for every $\sigma\in \cT_*^{[k]}$. Therefore
\[
\tau^{[k]}(x\ot y) = \tau^{[k]}(x\ot E_{\cJ^{[k]^*}}(y)),
\]
since $E_{\cJ^{[k]^*}}(y)$ is the unique element of minimal norm in the closure of the convex hull of $\{\sigma(y)\mid \sigma\in\cT_*^{[k]}\}\subset L^2(N^{[k]^*},\tau^{[k]^*})$. Since $\cZ_{k-1}$ and $\cJ^{[k]^*}$ are identified, $\tau^{[k]}|_{N\ot 1^{\ot 2^k-1}} = \tau\ot 1^{\ot 2^k-1}$, and $\tau^{[k]}|_{1\ot N^{\ot 2^k-1}} = 1\ot \tau^{[k]^*}$, we thus get
\begin{align}\label{eq:NotNstar}
\begin{split}
	\tau^{[k]}(x\ot y) &= \tau^{[k]}(x\ot E_{\cJ^{[k]^*}}(y))\\
	&= \tau^{[k]}(xE_{\cJ^{[k]^*}}(y)\ot 1^{\ot 2^k-1})\\
	&= \tau(xE_{\cJ^{[k]^*}}(y))\\
	&= \tau(E_{\cZ_{k-1}}(x)E_{\cJ^{[k]^*}}(y)).
	%&= \tau^{[k]}(E_{\cZ_{k-1}}(x)\ot E_{\cJ^{[k]^*}}(y))\\
	%&= (\tau\ot_{\cZ_{k-1}}\tau^{[k]^*})(x\ot y).
\end{split}
\end{align}
\end{proof}

Similar to the commutative setting in \cite{host-kra}, we will next characterize the subsystem $\cZ_k$ in terms of a Gowers--Host--Kra seminorm, which we first extend to the von Neumann algebraic framework. For $\epsilon\in V_k$, we write $\abs{\epsilon}\coloneqq \epsilon_1+\epsilon_2+\ldots+\epsilon_k$.

\begin{definition}\label{def-seminorm}
    Fix $k\geq 1$, and $x\in N$. We define
	\[
	\tnorm{x}_k \coloneqq \tau^{[k]}\Big(\bigotimes_{\epsilon\in V_k} C^{\abs{\epsilon}}(x)\Big)^{\!1/2^k},
	\]
	where $C(x)=x^*$ is the adjoint operator, that is,  $C^n(x)=x$ if $n$ is even and $C^n(x)=x^*$ if $n$ is odd.
\end{definition}

First of all, we note that by the definition of $\tau^{[k]}$,
\[
\tau^{[k]}\Big(\bigotimes_{\epsilon\in V_k} C^{\abs{\epsilon}}(x)\Big) = \tau^{[k-1]}\Big(E_{\cI^{[k-1]}}\Big(\bigotimes_{\eta\in V_{k-1}} C^{\abs{\eta}}(x)\Big)E_{\cI^{[k-1]}}\Big(\bigotimes_{\eta\in V_{k-1}} C^{\abs{\eta}}(x)^*\Big)\Big) \geq 0,
\]
and thus $\tnorm{x}_k$ is well-defined.

\begin{lemma}[\!\!{\cite[Lemma 3.9]{host-kra}}]\label{lem:proptnorm}
	Fix $k\geq 1$. The following hold.
	\begin{itemize}
		\item[(i)] When for each $\epsilon\in V_k$, we have $x_\epsilon\in N$, then
		\[
		\Big|\tau^{[k]}\Big(\bigotimes_{\epsilon\in V_k} x_\epsilon\Big)\Big| \leq \prod_{\epsilon\in V_k} \tnorm{x_\epsilon}_k.
		\]
		\item[(ii)] $\tnorm{.}_k$ is a seminorm on $(N,\tau)$.
		\item[(iii)] For any $x\in N$, $\tnorm{x}_k\leq \tnorm{x}_{k+1}$.
	\end{itemize}
\end{lemma}
\begin{proof}
The proof of \cite[Lemma 3.9]{host-kra} applies mutatis mutandis.
\end{proof}

We will now establish the following characterization of $\cZ_{k-1}$ in terms of the seminorm $\tnorm{.}_k$.

\begin{lemma}[\!\!{\cite[Lemma 4.3]{host-kra}}]\label{lem:normchar}
	Fix $x\in N$. Then the following are equivalent.
	\begin{enumerate}
		\item $E_{\cZ_{k-1}}(x) = 0$.
		\item $\tnorm{x}_k = 0$.
	\end{enumerate}
\end{lemma}
\begin{proof}
The proof proceeds essentially as in \cite[Lemma 4.3]{host-kra}, but we provide the details for the reader's convenience.

$(1)\Rightarrow (2)$: Assume $E_{\cZ_{k-1}}(x) = 0$. Using Lemma~\ref{lem:NotNstar} together with the definition of $\tnorm{x}_k$, we immediately get that $\tnorm{x}_k=0$.

$(2)\Rightarrow (1)$: Assume $\tnorm{x}_k=0$. By Lemma \ref{lem:proptnorm}(a), given any $x_\epsilon\in N$ for $\epsilon\in V_k^*$, 
\[
\Big|\tau^{[k]}\Big(x\ot\bigotimes_{\epsilon\in V_k^*} x_\epsilon\Big)\Big| \leq \tnorm{x}_k\cdot\prod_{\epsilon\in V_k^*} \tnorm{x_\epsilon}_k = 0.
\]
By density, it follows that
\[
|\tau^{[k]}(x\ot y)| = 0,
\]
for every $y\in N^{[k]^*}$, and in particular for every $y\in \cJ^{[k]^*}$. Through the identification of $\cJ^{[k]^*}$ and $\cZ_{k-1}$, this implies that $\tau(xw)=0$ for every $w\in \cZ_{k-1}$ and hence $E_{\cZ_{k-1}}(x)=0$.
\end{proof}

\begin{theorem}\label{thm:cZcompact}
For every $k\geq 1$, $\cZ_k$ is a compact extension of $\cZ_{k-1}$.
\end{theorem}
\begin{proof}
We will work on the level of the $L^2$-spaces. Consider $\cAP^2\coloneqq\cAP^2_{\cZ_{k-1},N}$, the subspace of relatively almost periodic functions relative to $\cZ_{k-1}$. We need to show that $\widehat{\cZ_k}\subset \cAP^2$.

Fix $x\in N$ and assume $E_{\cAP^2}(\hat x)=0$. Consider $y\coloneqq \bigotimes_{\epsilon\in V_k} C^{\abs{\epsilon}}(x)\in N^{[k]}$. Then 
\[
E_{\cAP^2\ot L^2(N^{[k]^*})}(\hat y)=0,
\]
and hence by Proposition \ref{prop:invJoin} also
\[
E_{\cI^{[k]}}(y)=0.
\]
Hence by the definition of the seminorm $\tnorm{x}_{k+1}$ and of $\tau^{[k+1]}$, we get
\[
\tnorm{x}_{k+1}^{2^{k+1}} = \tau^{[k+1]}\Big(\bigotimes_{\epsilon\in V_{k+1}} C^{\abs{\epsilon}}(x)\Big) 
= \tau^{[k]}(E_{\cI^{[k]}}(y)E_{\cI^{[k]}}(y)^*) = 0.
\]
Applying Lemma \ref{lem:normchar}, we thus get $E_{\cZ_k}(x)=0$, which finishes the proof of the theorem.
\end{proof}

We finish this section by noting that in general, $\cZ_{k}$ is not the maximal compact extension of $\cZ_{k-1}$, since this is already the case in the commutative setting (see \cite{host-kra}). However, it is always true that $\cZ_1$ is the maximal compact subsystem of the original system. The above therefore gives an alternative construction of this maximal compact subsystem.

\begin{proposition}\label{prop:Z1max}
$\cZ_1$ is the maximal compact subsystem of the ergodic action $\Gamma\act^\sigma (N,\tau)$.
\end{proposition}
\begin{proof}
Denote by $\cZ\subset N$ the maximal compact subsystem of $\Gamma\act (N,\tau)$, which exists by Proposition \ref{prop:quasireg}. Theorem \ref{thm:cZcompact} implies that $\cZ_1\subset \cZ$. For the converse, take $0\neq x\in \cZ$. It then suffices to show that $E_{\cZ_1}(x)\neq 0$. 
We can compute
\begin{align}
\begin{split}\label{eq:tnormx}
\tnorm{x}_2 &= \tau^{[2]}\Big(\bigotimes_{\epsilon\in V_{2}} C^{\abs{\epsilon}}(x)\Big)\\
&= \tau^{[1]}(E_{\cI^{[1]}}(x\ot x^*)E_{\cI^{[1]}}(x\ot x^*)^*)\\ 
&= (\tau\ot\tau)(E_{\cI^{[1]}}(x\ot x^*)E_{\cI^{[1]}}(x\ot x^*)^*)
\end{split}
\end{align}
We have that $K\coloneqq E_{\cI^{[1]}}(x\ot x^*)$ is the unique element of minimal norm in the closed convex hull of $\{(\sigma_\gamma\times\sigma_\gamma)(x\ot x^*)\mid \gamma\in\Gamma\}\subset L^2(N)\ot L^2(N)$. Now if $K=0$, we could repeat the second half of the proof of $(1)\Rightarrow (2)$ in Theorem \ref{thm:compchar} to conclude that $x=0$, which would contradict our assumption. Hence $K\neq 0$, and therefore by equation \eqref{eq:tnormx} also $\tnorm{x}_2\neq 0$. Applying Lemma \ref{lem:normchar}, we conclude that $E_{\cZ_1}(x)\neq 0$, which finishes the proof.
\end{proof}

Finally, we observe that we have now completed the proof of Theorem~\ref{thm-hk} from the Introduction:

\begin{proof}[Proof of Theorem~\ref{thm-hk}]
The Theorem follows from a combination of Lemma~\ref{lem:cZincreasing}, Lemma~\ref{lem:normchar}, Theorem~\ref{thm:cZcompact}, and Proposition~\ref{prop:Z1max}.
\end{proof}

\begin{remark}
    Suppose $\Gamma=\mathbb{Z}$ and $N=L^\infty(X,\mu)$ for a standard Borel probability space $(X,\mu)$. A deep structure theorem, independently obtained by Host and Kra \cite{host-kra} and Ziegler \cite{ziegler}, states that the commutative von Neumann algebra system $\cZ_k$ is isomorphic to a direct limit of commutative von Neumann algebra systems $M_n$ associated to translations on nilmanifolds $G_n/\Lambda_n$, where $G_n$ is a nilpotent Lie group of nilpotency class $k$ and $\Lambda_n\leq G_n$ is a cocompact lattice. An important step in the proof of the Host--Kra--Ziegler structure theorem is to identify the system $\cZ_k$ as an iterated tower of skew-products
$$
U_1 \rtimes_{\rho_1} U_2 \rtimes_{\rho_2} \ldots \rtimes_{\rho_{k-1}} U_{k}
$$
for some compact abelian groups $U_1,\dots,U_k$ and $U_i$-valued cocycles $\rho_i$ of type\footnote{The type of a cocycle is a cohomological property, roughly meaning that it lifts to a coboundary on the corresponding cubic system, e.g., for type $i+1$ to the cubic system $(N^{[i+1]},\tau^{[i+1]})$, see \cite[Chapter 18, \S 3]{hk-book} for a precise definition.} $i+1$, where the skew-product operation $\rtimes$ is performed from left to right, and $U_1$ has the structure of a Kronecker system, i.e., an ergodic rotation on the compact abelian group $U_1$. 

Host--Kra--Ziegler type structure theorems have been studied for other countable abelian groups $\Gamma$ as well, for example when $\Gamma=\mathbb{F}_p^\omega$ (the direct sum of infinitely many copies of a finite field of prime order $p$) by Bergelson, Tao, and Ziegler in \cite{btz}, when $\Gamma=\bigoplus_{p\in\mathcal{P}} \mathbb{F}_p$ for a multiset of primes $\mathcal{P}$ by Shalom in \cite{shalom3}, and when $\Gamma$ has bounded torsion by the first author, Shalom, and Tao in \cite{jst-tdsystems}. A satisfactory structure theorem for arbitrary countable abelian groups in the case of systems $\cZ_2$ of order $2$ was established  by the first author, Shalom, and Tao in \cite{jst}. Currently, we do not have a satisfactory Host--Kra--Ziegler type structure theorem for actions of arbitrary countable abelian groups for systems of higher order in the commutative setting. 
\end{remark}

\begin{remark}
Generalizing the above geometric representation theory of Host--Kra--Ziegler subsystems $\cZ_k$ to the non-commutative setting necessarily has to start with the maximal compact part of the system $\cZ_1$, which in the commutative setting corresponds to the Kronecker factor, that is, the subsystem generated by the eigenfunctions. In \cite{olesen}, Olesen, Pedersen, and Takesaki study ergodic actions of compact abelian groups on von Neumann algebras and obtain a classification of the corresponding systems in terms of $2$-cocycles on the dual group. After interpreting the results of \cite{olesen} for $\cZ_1$ for a discrete group action, one possible direction to explore the geometric representation theory of the higher-order subsystems $\cZ_k$ could be to study relative versions of the results in \cite{olesen} such as von Neumann algebra valued $2$-cocycles which could encode the higher-order eigenfunctions.   We hope to report tangible findings in future work. 
\end{remark}

\end{document}